\documentclass[11pt, a4paper]{amsart}
\usepackage{amsmath}
\usepackage{amssymb}
\usepackage{color}
\newtheorem{prop}{Proposition}[section]
\newtheorem{teo}{Theorem}[section]
\newtheorem{lema}{Lemma}[section]
\newtheorem{coro}{Corollary}[section]

\newtheorem{rem}{Remark}[section]

\def\ep{\varepsilon}

\def\a{\mathfrak a}
\def\R{\mathbb R}
\def\K{{\mathcal K}}
\def\A{{\mathcal A}}

\begin{document}

\title[Nonlocal diffusion
with absorption]{Large time  behavior for a nonlocal diffusion equation with
absorption and bounded initial data}

\author[J. Terra \and N. Wolanski]{Joana Terra\and Noemi Wolanski}

\address{Joana Terra\hfill\break\indent
Departamento  de Matem{\'a}tica, FCEyN \hfill\break\indent UBA (1428)
Buenos Aires, Argentina.} \email{{\tt jterra@dm.uba.ar} }

\address{Noemi Wolanski \hfill\break\indent
Departamento  de Matem{\'a}tica, FCEyN \hfill\break\indent UBA (1428)
Buenos Aires, Argentina.} \email{{\tt wolanski@dm.uba.ar} }

\date{}

\begin{abstract}
We study the large time behavior of nonnegative solutions of the
Cauchy problem $u_t=\int J(x-y)(u(y,t)-u(x,t))\,dy-u^p$,
$u(x,0)=u_0(x)\in L^\infty$, where $|x|^{\alpha}u_0(x)\rightarrow A>0$ as
$|x|\rightarrow\infty$. One of our main goals is the study of the critical case
$p=1+2/\alpha$ for $0<\alpha<N$, left open in previous articles, for which we prove that $t^{\alpha/2}|u(x,t)-U(x,t)|\to 0$ where $U$ is  the solution of the heat equation with absorption with  initial
datum $U(x,0)=C_{A,N}|x|^{-\alpha}$. Our proof, involving  sequences of rescalings of the solution, allows us to establish
also the large time behavior of solutions having more general
nonintegrable initial data $u_0$ in the supercritical case
and also
in the critical case ($p=1+2/N$)
for bounded and integrable $u_0$.
\end{abstract}

\maketitle

\date{}

\section{Introduction}
\label{Intro} \setcounter{equation}{0}

Consider the following nonlocal evolution problem with absorption
\begin{equation}\label{problem}
\begin{cases}
u_t =  \int J(x-y)\left(u(y,t)-u(x,t)\right)\,dy-u^p(x,t)& \mbox{in }
\ \R^N\times(0,\infty)\\
u(x,0) =  u_0(x)& \mbox{in }\ \R^N,
\end{cases}
\end{equation}
where $J\in C_0^{\infty}(\R^N)$ is radially symmetric, $J\geq 0$ with
$\int J=1$ and $u_0\ge0$ and bounded.

Equation \eqref{problem} can be seen as a model for the density of a
population at a certain point and given time. In fact, let $u$
represent such density and the kernel $J(x-y)$ represent the
probability distribution density of jumping from a point $x$ to a
point $y$. Then, using the symmetry of the kernel $J$, the diffusion
term $\int J(x-y)(u(y,t)-u(x,t))\,dy$ represents the difference
between the rate at which the population is arriving at the point $x$
and the rate at which it is leaving  $x$. The absorption term
$-u^p$ represents a rate of consumption due to an internal reaction.

This diffusion operator has been used to model several nonlocal
diffusion processes in the last few years. See for instance
\cite{BCh1,BCh2,BFRW,CF,F,Z}. In particular,
nonlocal diffusions are of interest in biological and biomedical
problems. Recently, these kind of nonlocal operators  have also been used  for
image enhancement  \cite{GO}.

\medskip
We are interested in the large time behavior of the solutions to
\eqref{problem} and how the space dimension $N$, the absorption
exponent $p$ and the assumptions on the initial data $u_0$
influence  the result.

These kind of problems have been widely studied for the heat equation with absorption or, more generally, the porous medium equation or other diffusion equations. (See, for instance, \cite{CQW,H,KP,KP2,KU,Za}).

\medskip

In the case of semilinear problems, one possible approach is the direct use of the variations of constants formula associated to the knowledge of a fundamental solution of the purely diffusive linear part.
This approach has been considered for problem \eqref{problem} in
\cite{PR,TW} and it
 allowed to study  the supercritical cases, in which the presence of
 the absorption
term does not influence the asymptotic behavior.

Another possible approach is to consider rescalings that leave the purely diffusive equation unchanged. This is done, for instance, for the porous medium or the evolutionary $p$-laplace equations.

At first sight, this last approach is not possible for \eqref{problem} since the nonlocal diffusion equation is not invariant under any rescaling. Anyhow, it is easy to see that if $u$ is a solution to \begin{equation}\label{equation}
u_t=\int J(x-y)(u(y,t)-u(x,t))\,dy = Lu,
\end{equation}
and, for $k>0$ and $f(k)$ any function of the parameter $k$,
$$
u^k (x,t)=f(k)u(k  x,k ^2 t),
$$
then $u^k $ is a solution to the following equation
\begin{equation}\label{equation-lambda}
v_t=k ^2\int J_k (x-y)\big(v(y,t)-v(x,t)\big)\,dy,
\end{equation}
where $J_k (x)=k ^N J(k  x)$. It is not difficult to prove that for a fixed smooth function $v$, the right hand side in \eqref{equation-lambda} converges, as $k$ goes to infinity, to $\a\,\Delta v$ where $\a$ is a constant that depends only on the kernel $J$ of the nonlocal operator.

This fact has already been used --with $\ep=k ^{-1}\to0$-- in order to prove that the solutions of the rescaled problems set in a fixed bounded domain converge to the solution of the heat equation with diffusivity $\a$. (See, for instance, \cite{CER,CERW2}).

\medskip

Therefore, since $k$ going to infinity  for $u^k(x,1)$ amounts to $t$ going to infinity for $u(x,t)$, it is not at all striking that the asymptotic behavior as $t$ goes to infinity of the solution of \eqref{problem} is the same as that of the solution of the equation obtained by replacing the nonlocal operator by  $\a\,\Delta$, as was proved in \cite{PR} when $u_0\in L^\infty\cap L^1$ and $p>1+2/N$ and in \cite{TW}
when $u_0$ is bounded and $|x|^\alpha u_0(x)\to A>0$ as $|x|\to\infty$ with $0<\alpha\le N$ and $p> 1+2/\alpha$.

Both in \cite{PR} and \cite{TW} no rescaling was
proposed and instead, the variations of constants formula was used.
Nevertheless, such method only led  to the study of the
  supercritical cases.

\medskip

As stated above, the idea of the rescaling method is that the behavior of $u(x,t)$ as $t\to\infty$ is that of $u^k(x,1)$ as $k \to\infty$ (for a suitable choice of $f(k)$).

In order to prove the convergence of the functions $u^k (x,1)$ on
compact sets of $\R^N$,
it is necessary to establish compactness of a family
of uniformly bounded solutions to the  equation satisfied by $u^k$.
 In the case of the heat or the porous medium equations, this compactness follows from their regularizing effect.

\medskip

The purpose of this paper is twofold. On one hand, to establish the
large time behavior in the critical case $p=1+2/\alpha$, $0<\alpha<N$
that was left open in \cite{TW}. On the other hand, to show how to use
the rescaling method in the present situation in which each $u^k $ is
a solution of a different equation.  And moreover, how to obtain, in the present situation, the compactness of the family $\{u^k\}$ from uniform $L^\infty$ bounds.

One very important issue that we had to overcome is the lack of a regularizing effect of
equation \eqref{equation} and its rescalings. In fact,
$u(x,t)$ is exactly as smooth as $u_0(x)$. This problem is overcome by the observation that
the fundamental solution of equation \eqref{equation} can be  decomposed as $e^{-t}\delta+W(x,t)$ with $W(x,t)$ smooth, as proved in \cite{ChChR}.
Then, the variations of constants formula for the rescaled equations allowed us to decompose $u^k(x,t)=v^k(x,t)+h^k(x,t)$ with $v^k\to0$ and $h^k$ smooth. So, the limit as $k\to\infty$ of $u^k(x,t)$ for $t>0$ is that of the smooth functions $h^k$, for which we can prove compactness by establishing uniform Holder estimates.

One of the main contributions of the present paper is a very sharp
estimate on the space-time  behavior of $W(x,t)$ (as well as its space
and time
derivatives). This estimate is invariant under the  rescaling $W_k (x,t)=k ^NW(k  x,k ^2t)$,
thus leading to estimates for the rescalings $u^k$ of the solution $u$.

This study of the good part of the fundamental solution and the idea of how to use the rescaling method in order to study asymptotics  related to the nonlocal diffusion equation \eqref{equation} give insight into how to attack other semilinear problems related to this equation like blow up or quenching problems (determine blow up and quenching profiles, for instance)  and thus they are of an independent interest.

\medskip

In the present paper we apply this method in order to study the asymptotic behavior as time goes to infinity of the solution to \eqref{problem} for general bounded initial data $u_0$. We find that
this behavior is determined by the way $u_0$ decays at infinity, retrieving for \eqref{problem} results that were known for the heat equation (see \cite{H,KP,KU}). In particular, this method allows to treat the critical case $p=1+2/\alpha$, $0<\alpha<N$ left open in our previous paper \cite{TW}. Moreover, we also complete the results of \cite{PR} for integrable initial data by proving that --in the critical case $p=1+2/N$-- there holds that $t^{N/2}u(x,t)\rightarrow0$ as $t\to\infty$ (as compared to the supercritical case $p>1+2/N$ where a nontrivial limit is achieved).

\medskip

 Moreover, since our approach allows for very general
 initial data, we find results for both integrable and  nonintegrable
 initial data that do not behave as a negative power at infinity, and
 we give some examples of application of our results
to such cases
at the end of this article.

Before presenting some examples of initial data to which our results apply let us introduce some notation.

\noindent{\bf Notation.} We consider  rescalings that depend on the behavior of $u_0$ at infinity.
For $k>0$ we denote by $v_k(x,t)=k^Nv(kx,k^2t)$ and, following the notation in \cite{KU}, we denote  by $v^k(x,t)=f(k)v(kx,k^2t)$ where
 $$f(k):=\frac{k^N}{\int_{B_k}u_0}.$$

\begin{rem}Since $u_0\ge 0$, $u_0\neq0$, there exist $\kappa>0$ and $x_0\in
\R^N$ such that $\int_{B_k(x_0)}u_0(x)\,dx\ge\kappa k^N$ for
$k>0$ small. Without loss of generality we will assume that $x_0=0$.
\end{rem}

We assume further.

\noindent{\bf Conditions on $f$}
\begin{enumerate}
\item[{\bf (F1)}] $u_0\in L^\infty(\R^N)$ and there exists $B>0$ such that
$f(|x|)u_0( x)\le B$.

\bigskip

\item[{\bf (F2)}] For every $\delta>0$, there exists $C_\delta>0$ such that $f(k)\le C_\delta f(l)$ if $k_0\le
k\le l\delta^{-1}$.

\bigskip

 \item[{\bf (F3)}] There exists $c_0\ge0$ such that $F(k):=f(k)^{1-p}k^2\to
c_0$ as $k\to\infty$.

\end{enumerate}

\medskip

\begin{rem} Assumption {\bf (F3)} implies that $f(k)\ge c_1k^{2/(p-1)}$
if $k\ge k_0$ and $k_0$ is large. Therefore, $f(k)\to\infty$ as $k\to\infty$.
\end{rem}

\bigskip

\noindent{\bf Examples}
By including the function $f$ in our rescaled sequence we are able to
deal with rather general initial conditions.
It is also interesting to note that, for $u_0\in L^\infty$
satisfying
\begin{equation}\label{condu0}
|x|^\alpha u_0(x)\to A>0\qquad\mbox{as }|x|\to\infty\mbox{ \ with }0<\alpha<N,
\end{equation}
the function $f$ behaves like  $k^\alpha$ as $k$ tends to
infinity. This is then, a rather usual rescaling in this case.

\bigskip

The following is a list of  examples of initial
data satisfying our assumptions.

\begin{enumerate}
\item[{\bf Ex. 1}] Assume $u_0\in L^\infty(\R^N)$ and $|x|^\alpha u_0(x)\to A>0$ as $|x|\to\infty$ with $0<\alpha<N$. Then,
\begin{align*}
&f(k)\sim k^\alpha\qquad\mbox{so we take it to be equal}\\
&F(k)=k^{-\alpha(p-1)+2}\to \begin{cases}0\quad&\mbox{if }p> 1+\frac2\alpha\\
1\quad&\mbox{if }p= 1+\frac2\alpha
\end{cases}
\end{align*}

\medskip

\item[{\bf Ex. 2}] Assume $u_0\in L^\infty(\R^N)$ and $|x|^N u_0(x)\to A>0$ as $|x|\to\infty$. Then,
\begin{align*}
&f(k)\sim \frac{k^N}{\log k}\qquad\mbox{so we take it to be equal}\\
&F(k)={k^{-N(p-1)+2}}(\log k)^{p-1}\to 0\in\R\qquad\mbox{if }p> 1+\frac2N.
\end{align*}

\medskip

\item[{\bf Ex. 3}] Assume $u_0\in L^\infty(\R^N)$ and $\displaystyle\frac{|x|^\alpha}{\log|x|} u_0(x)\to A>0$ as $|x|\to\infty$ with $0<\alpha<N$. Then,
\begin{align*}
&f(k)\sim \frac{k^\alpha}{\log k}\qquad\mbox{so we take it to be equal}\\
&F(k)=k^{-\alpha(p-1)+2}(\log k)^{p-1}\to 0\in\R\qquad\mbox{if }p> 1+\frac2\alpha.
\end{align*}

\medskip

\item[{\bf Ex. 4}] Assume $u_0\in L^\infty(\R^N)$ and $\displaystyle{|x|^\alpha}(\log|x|) u_0(x)\to A>0$ as $|x|\to\infty$ with $0<\alpha<N$. Then,
\begin{align*}
&f(k)\sim {k^\alpha}{\log k}\qquad\mbox{so we take it to be equal}\\
&F(k)=\frac{k^{-\alpha(p-1)+2}}{(\log k)^{p-1}}\to 0\quad\mbox{if }p\ge 1+\frac2\alpha.
\end{align*}

\medskip

\item[{\bf Ex. 5}] Assume $u_0\in L^\infty(\R^N)$ and $\displaystyle\frac{|x|^N}{\log|x|} u_0(x)\to A>0$ as $|x|\to\infty$. Then,
\begin{align*}
&f(k)\sim \frac{k^N}{\log^2 k}\qquad\mbox{so we take it to be equal}\\
&F(k)=k^{-N(p-1)+2}(\log k)^{2(p-1)}\to 0\quad\mbox{if }p> 1+\frac2N.
\end{align*}

\medskip

\item[{\bf Ex. 6}] Assume $u_0\in L^\infty(\R^N)$ and $\displaystyle{|x|^N}(\log|x|) u_0(x)\to A>0$ as $|x|\to\infty$. Then,
\begin{align*}
&f(k)\sim \frac{k^N}{\log\log k}\qquad\mbox{so we take it to be equal}\\
&F(k)=k^{-N(p-1)+2}(\log\log k)^{p-1}\to 0\quad\mbox{if }p> 1+\frac2N.
\end{align*}

\medskip

\item[{\bf Ex. 7}] Assume $u_0\in L^1(\R^N)$. Then,
\begin{align*}
&f(k)\sim k^N\qquad\mbox{so we take it to be equal}\\
&F(k)=k^{-N(p-1)+2}\to \begin{cases}0\quad&\mbox{if }p> 1+\frac2N\\
1\quad&\mbox{if }p= 1+\frac2N
\end{cases}
\end{align*}
\end{enumerate}

\bigskip

Let us just mention that the function $f(k)$ is related to the rate of
decay in time of the solution whereas $c_0=\lim_{k\to\infty}F(k)$
turns  out to be  the coefficient in front of the absorption term in the equation satisfied by the limiting profile.

\bigskip

For our main results we refer to Section 4.

\bigskip

The paper is organized as follows. In Section 2 we construct barriers
for the good part $W$ of the fundamental solution  of \eqref{equation}, as well as for
its space and time derivatives. By using these barriers we obtain an
upper bound for the solution $u_L$ to \eqref{equation} for all
times. This result improves the one we had established for finite time
intervals in \cite{TW} in case $u_0$ satisfies \eqref{condu0}. Also, this bound on $u_L$ implies that the rescaled functions $u^k$ are uniformly bounded in $\R^N\times[\tau,\infty)$ for every $\tau>0$.

In Section 3 we analyze the rescaled problems satisfied by  the $u^k$'s and prove that, under sequences, $\{u^k\}$  converges uniformly on
compact sets to a solution $U$ of the equation $U_t-{\a}\Delta U=-c_0U^p$, where
$c_0=\lim_{k\rightarrow\infty}F(k).$

Moreover, we obtain a general result stating that
 we can determine the limit function $U$. In
fact, we prove that if $u_0^k(x)\rightarrow\phi(x)$ as $k$ tends to infinity in the
sense of distributions and, for every $R>0$ there holds that
$$||u^k||_{L^1(B_R\times(0,\tau))}\leq C(R)\tau $$
 and either  there exists $\gamma>0$ such that
$$||u^k||_{L^p(B_R\times(0,\tau))}\leq C(R)\tau^{\gamma}$$
or, for every $\tau\le1$, $R>0$,
$$
\lim_{k\to\infty} ||u^k||_{L^p(B_R\times(0,\tau))}=0,$$
then $U$ satisfies moreover,
$$U(x,0)=\phi(x).$$

Then we establish, for the different cases of nonintegrable initial data $u_0$ considered in the examples, the
desired $L^1$ and $L^p$ bounds.

At the end of the section we analyze
the case where $u_0$ is integrable (in which case it does not satisfy an $L^p$
estimate as above) and determine $U$, both for $p$ supercritical and for $p$ critical.

Finally, in Section 4 we prove our main results on the asymptotic
behavior of the solution for initial data comprised in the examples
above. In particular, we obtain the behavior in the critical cases $p=1+2/\alpha$ when $u_0$ satisfies \eqref{condu0}, and $p=1+2/N$ when $u_0\in L^1\cap L^\infty$, that were left open in \cite{TW} and \cite{PR} respectively.

\bigskip

\section{Barriers and first asymptotic estimates}

In this section we analyze in detail the smooth part W of the
fundamental solution to the linear equation
\begin{equation}\label{eqL2}
u_t=\int J(x-y)(u(y,t)-u(x,t))\,dy = Lu.
\end{equation}
We establish upper bounds for W and also for its space and time
derivatives. These will be essential in Section 3, in order to prove the convergence of the rescaled sequence.

In \cite{ChChR} the authors observed that the fundamental
solution of \eqref{eqL2} can be written as
$$
U(x,t)=e^{-t}\delta +W(x,t)
$$
where $\delta$ is the Dirac measure and $W(x,t)$ is a smooth function. Then, in \cite{IR} the authors showed that
$$
|W(x,t)|\le C_0t^{-N/2}.
$$

In \cite{TW} we established that $W$ is a solution to the problem
\begin{equation}\label{eq-W}
\left\{\begin{aligned}
&W_t(x,t)=\int J(x-y)\big(W(y,t)-W(x,t)\big)\,dy + e^{-t}J(x)\\
&W(x,0)=0,
\end{aligned}\right.
\end{equation}
and used this fact to prove that $W\ge0$ and to obtain estimates in
$L^q(\R^N)$, which in turn allowed us to establish the asymptotic
behavior of solutions to the nonlocal problem with absorption in the
supercritical case.

In order to deal with the critical case, we will use the method of rescaled sequences for which we need the knowledge of the behavior of $W(x,t)$ as $|x|\to\infty$. We obtain this behavior from
sharp barriers. We have,
\begin{teo}\label{teo-barrier-W}
Let $W$ as above. There exists a constant $C>0$ depending only on $J$ and $N$ such that
\begin{equation}\label{barrier-W}
W(x,t)\le C\frac t{|x|^{N+2}}.
\end{equation}
\end{teo}
\begin{proof}
First observe that $W(x,t)\le v_1(x,t):=\|J\|_\infty t$. In fact, in any finite time interval, the function
$v_1$ is a  bounded supersolution of the problem \eqref{eq-W} satisfied by $W$. Thus, the inequality follows from the comparison principle.

Now, let $v_2(x,t)= C\frac t{|x|^{N+2}}$. We will show that there is a constant $C$ depending only on $J$ and $N$ such that $v_2$ is a supersolution of the Dirichlet problem
$$
\left\{\begin{aligned}
&v_t-Lv=e^{-t}J(x)\qquad\mbox{in }\A:=\{|x|\ge K\sqrt t\}\cap\{|x|\ge2R\}\\
&v= W \qquad\qquad\quad\ \  \ \ \ \  \mbox{in the complement of }\A\\
&v(x,0)=0
\end{aligned}\right.
$$
satisfied by $W$.
Here $R$ is large enough and such that $B_R$ contains the support of $J$, and $K$ is a large enough constant to be determined.

In fact, $v_2\ge W$ in $\A^c$ if $C$ is large since $W\le \|J\|_\infty t$ and $W\le C_0 t^{-N/2}$.

On the other hand, ${v_2}_t= \frac C{|x|^{N+2}}$. In order to estimate $L v_2$ we use Taylor's expansion to get,
$$\begin{aligned}
&\frac1{|y|^{N+2}}-\frac1{|x|^{N+2}}=-\frac{N+2}{|x|^{N+4}}\,x\cdot(y-x)-
\frac12\frac{N+2}{|x|^{N+4}}\,
|y-x|^2\\
&\hskip1cm\ \
+\frac12\frac{(N+2)(N+4)}{|x|^{N+6}}\,\big|x\cdot(y-x)\big|^2+
\int_0^1O\Big(\frac{|y-x|^3}{|x+s(y-x)|^{N+5}}\Big)\,ds.
\end{aligned}
$$

Now, since $J$ is radially symmetric,

\begin{align*}
 -&\frac{N+2}{|x|^{N+4}}\sum_{i=1}^N x_i\int J(x-y)(y_i-x_i)\,dy=0,\\
\frac12&\frac{N+2}{|x|^{N+4}}\int J(x-y)|y-x|^2\,dy=\frac{(N+2)N}{|x|^{N+4}}\,\a,
\end{align*}
and $$\frac12\sum_{i,j=1}^N\frac{x_ix_j}{|x|^{N+6}}\big(-(N+2)(N+4)\big)\int J(x-y)(y_i-x_i)(y_j-x_j)\,dy= -\frac{(N+2)(N+4)}{|x|^{N+4}}\,\a,$$
where $\a=\frac1{2N}\int J(x)|x|^2\,dx$.

On the other hand, $x+s(y-x)\ge |x|-|y-x|\ge \frac12|x|$ if $|x|\ge 2R$ and $|y-x|\le R$. Thus, if $|x|\ge 2R$,
$$
\int\int_0^1
J(x-y)O\Big(\frac{|y-x|^3}{|x+s(y-x)|^{N+5}}\Big)\,ds\,dy\le \frac
{C_1}{|x|^{N+5}}.
$$

Putting everything together we get, if $|x|\ge 2R$ and $R$ is large enough,
$$
{v_2}_t-L v_2\ge \frac{C}{|x|^{N+2}}\Big(1-{C_2}\frac t{|x|^2}\Big).
$$

So that, if $|x|\ge K\sqrt t$ with $K$ large enough,
$$
{v_2}_y-L v_2\ge \frac{C/2}{|x|^{N+2}}
\ge e^{-t}J(x)
$$
if $C$ is large enough since $J$ is bounded and has compact support.
\end{proof}

Now we find a barrier for the space derivatives of $W$.

\begin{teo} Let $W$ be as above. There exists a constant $C>0$ depending only on $J$ and $N$ such that
\begin{equation}\label{barrier-nablaW}
|\nabla W(x,t)|\le C\frac t{|x|^{N+3}}.
\end{equation}
\end{teo}
\begin{proof}
We proceed as above. First, by differentiating the equation satisfied by $W$ we find that $V_i=W_{x_i}$ is the solution to
\begin{equation}\label{eq-Vi}
\left\{\begin{aligned}
& {V_i}_t-L V_i=e^{-t}J_{x_i}\\
&V_i(x,0)=0
\end{aligned}\right.
\end{equation}

As with $W$ we find immediately that $|V_i|\le \|\nabla J\|_\infty t$. On the other hand, from the Fourier characterization of $W$ it can be seen that $\|\nabla W(\cdot,t)\|_\infty\le C_0 t^{-\frac{N+1}2}$. (See, for instance, \cite{IR}).
Therefore, for every $K>0$ there exists a constant $C$ such that
$$
\bar v(x,t):= C\frac t{|x|^{N+3}}\ge V_i\qquad\mbox{in }\A^c.
$$

On the other hand, the same type of computation as the one in Theorem \ref{teo-barrier-W} yields that, if $K$ is large enough, there exists $C$ such that $\bar v$ is a supersolution to the Dirichlet problem
$$
\left\{\begin{aligned}
&v_t-Lv=e^{-t}J_{x_i}(x)\qquad\mbox{in }\A:=\{|x|\ge K\sqrt t\}\cap\{|x|\ge2R\}\\
&v=V_i \qquad\qquad\quad\ \hskip1cm \mbox{in the complement of }\A\\
&v(x,0)=0
\end{aligned}\right.
$$
satisfied by $V_i$.

We conclude that $V_i\le \bar v$.

Analogously, $-\bar v$ is a subsolution to this problem. Thus,
$
|V_i|\le \bar v
$
and the theorem is proved.
\end{proof}

The estimate in \eqref{barrier-nablaW} allows to derive estimates of the $L^1$ norm of $\nabla W$. In fact,
$$
\begin{aligned}
\int |\nabla W(x,t)|\,dx&=\int_{|x|\le \sqrt t}|\nabla W(x,t)|\,dx+
\int_{|x|\ge \sqrt t}|\nabla W(x,t)|\,dx\\
&\le C_0\int_{|x|\le \sqrt t}t^{-\frac{N+1}2}\,dx+C\int_{|x|\ge \sqrt t}
\frac t{|x|^{N+3}}\,dx\\
&=C_{N,J}\Big[t^{-\frac{N+1}2}t^{\frac N2}+t t^{-\frac32}\Big]\\
&=C_{N,J} t^{-\frac12}.
\end{aligned}
$$

So, we obtain a first estimate that is suitable for large times
\begin{equation}\label{first-int-nablaW}
\int |\nabla W(x,t)|\,dx\le C_1 \,t^{-\frac12}
\end{equation}
with $C_1$ depending only on $J$ and $N$.

On the other hand, since $|\nabla W|\le Ct$, there holds that
\begin{equation}\label{otra-barrera-nablaW}
|\nabla W(x,t)|\le C\frac t{(1+|x|)^{N+3}}.
\end{equation}
Thus, we have the following estimate, which is better than the
previous one for small $t$.
\begin{equation}\label{second-int-nablaW}
\int |\nabla W(x,t)|\,dx\le C_1\,t.
\end{equation}

\bigskip

We obtain similar results for $W_t$. We have,
\begin{teo} Let $W$ be as above. There exists a constant $C>0$ depending only on $J$ and $N$ such that
\begin{equation}\label{barrier-Wt}
| W_t(x,t)|\le e^{-t}J(x)+C\frac t{(1+|x|)^{N+4}}.
\end{equation}
\end{teo}
\begin{proof}By differentiating the equation satisfied by $W$ we obtain
$$
(W_t)_t(x,t)=L W_t(x,t)-e^{-t}J(x).
$$

On the other hand, from the equation for $W$ we get
$$
W_t(x,0)=J(x).
$$

Let $V(x,t)=W_t(x,t)-e^{-t}J(x)$. Then,
\begin{equation}\label{V}
\left\{\begin{aligned}
& V_t-LV=e^{-t}(J*J-J)\\
&V(x,0)=0
\end{aligned}\right.
\end{equation}

Since $|J*J-J|\le 2\|J\|_\infty$ we get a first estimate for $V$: $|V(x,t)|\le 2\|J\|_\infty t$.

We can derive the estimate $|V(x,t)|\le C t^{-\frac{N+2}2}$ by differentiating $W$ with respect to time in its Fourier representation.

Now, proceeding as  above we see that there exist $C$ and $K$ large so that the function $C\frac t{|x|^{N+4}}$ is a supersolution of the following problem satisfied by $V$.
$$
\left\{\begin{aligned}
&v_t-Lv=e^{-t}(J*J-J)\qquad\mbox{in }\A:=\{|x|\ge K\sqrt t\}\cap\{|x|\ge2R\}\\
&v= V \qquad\qquad\quad\ \hskip2cm \mbox{in the complement of }\A\\
&v(x,0)=0
\end{aligned}\right.
$$

Analogously, $-C\frac t{|x|^{N+4}}$ is a subsolution to this problem.
 Therefore,
$$
|V(x,t)|\le C\frac t{|x|^{N+4}}.
$$

So that,
$$
|W_t(x,t)|\le e^{-t}J(x)+|V(x,t)|\le e^{-t}J(x)+C\frac t{|x|^{N+4}}.
$$

Since, $|V(x,t)|\le Ct$ we prove \eqref{barrier-Wt}.
\end{proof}

From estimate \eqref{barrier-Wt} we obtain the following estimates
\begin{align}\label{integral-Wt}
&\|W_t(\cdot,t)\|_{L^1(\R^N)}\le C t^{-1}\\
&\|W_t(\cdot,t)\|_{L^1(\R^N)}\le e^{-t}+ C t.
\end{align}

\bigskip

Finally, we construct a barrier for the solution $u_L$. We have,
\begin{prop}\label{barrier-uL} Let $u_L$ be the solution to
\eqref{eqL2} with initial datum $u_0$ satisfying our
assumptions. There exists a constant $C$ depending only on
$\|u_0\|_\infty$, $B,N,J$ such that
\begin{equation}\label{eq-barrier-uL}
\begin{aligned}
&f(t^{1/2} )u_L(x,t)\le C,\\
&f(|x|)u_L(x,t)\le  C.
\end{aligned}
\end{equation}

In particular, if $\mu>0$ is such that $k^\mu\le C f(k)$ if $k\ge1$ there holds that
\begin{equation}\label{bound-uL}
u_L(x,t)\le \frac C{(1+t^{1/2}+|x|)^\mu}
\end{equation}

Let us recall that this is always the case when $\mu=\frac2{p-1}$.
\end{prop}

\begin{rem} \eqref{bound-uL} improves the estimate found in \cite{TW}, Proposition 2.1 where the estimate was proved in finite time intervals when $u_0$ satisfies \eqref{condu0} with $\alpha>0$.
\end{rem}

\begin{proof} We already know that $u_L(x,t)\le \|u_0\|_\infty$.
Let us begin with the time estimate.  We have --by the estimates on
$W$ and our assumptions on $u_0$-- that for $t$ large,
$$
\begin{aligned}
&f(t^{1/2})u_L(x,t)=f(t^{1/2})e^{-t}u_0(x)+f(t^{1/2})\int
W(x-y,t)u_0(y)\,dy\\
&\ \ \le C +\frac1{\int_{B_{t^{1/2}}}u_0}\int_{|y|<t^{1/2}}t^{N/2}W(x-y,t)
u_0(y)\,dy +\int_{|y|>t^{1/2}}W(x-y,t)f(t^{1/2})u_0(y)\,dy \\
&\ \ \le C\Big(1+\int W(x-y,t)\,dy\Big)\le C.
\end{aligned}
$$

On the other hand, since $u_L$ is bounded and $f$ is locally
bounded,
$$
f(t^{1/2})u_L(x,t)\le C
$$
if $t$ is bounded.

Now, we estimate,
$$
\begin{aligned}
&f(|x|) u_L(x,t)=e^{-t}f(|x|) u_0(x)+f(|x|)\int W(x-y,t)u_0(y)\,dy\\
&\ \ \le B+f(|x|)\int_{|y|<\frac12|x|}W(x-y,t)u_0(y)\,dy+f(|x|)\int_{|y|\ge\frac12|x|}W(x-y,t)u_0(y)\,dy\\
&\ \ =B+I+II.
\end{aligned}
$$

Observe that $f(2k)=\frac {2^Nk^N}{\int_{B_{2k}}u_0}\le 2^Nf(k)$.
Thus,
$$
II\le C \int W(x-y,t)f(2|y|)u_0(y)\,dy\le C2^N B.
$$

In order to estimate $I$ we use the barrier of $W$. We have, since $|y|<\frac12|x|$ implies that $|x-y|>\frac12|x|$,
\begin{align*}
I&=f(|x|)\int_{|y|<\frac12|x|}W(x-y,t)u_0(y)\,dy\le C|x|^{-2} t\frac1{\int_{B_{|x|}}u_0}
\int_{|y|<\frac12|x|}u_0(y)\,dy\\
&\le  C\frac t{|x|^2}\le C\quad\mbox{if } |x|^{2}>t.
\end{align*}

On the other hand,  in the region $k_0^2\le |x|^2\le  t$   there
holds that
$$
f(|x|) u_L(x,t)\le C_1 f(t^{1/2})u_L(x,t)\le C.
$$

Finally, if $|x|\le k_0$, there holds that $f(|x|)u_L(x,t)\le C$.
So, the proposition is proved.\end{proof}

\begin{rem}\label{rmkulequL}
When $u_0\geq 0$ the solution $u_L$ of the homogeneous equation
\eqref{eqL2} with initial data $u_0$ is non-negative. Thus, $u_L$ is a
supersolution to \eqref{equation} and $0$ is a subsolution to
\eqref{equation}. By the comparison principle we deduce that
$$0\leq u(x,t)\leq u_L(x,t),$$
for every solution $u$ of \eqref{equation}. Hence, the estimates of the
previous proposition hold with $u_L$ replaced by $u$, that is,
$$f(|x|)u(x,t)\leq C\quad{\rm and }\quad  f(t^{1/2})u(x,t)\leq C.$$
\end{rem}

\bigskip

\section{The rescaled problem}
\setcounter{equation}{0}

In this section we analyze the rescaled problem. This is, the one
satisfied by the rescaled functions $u^k$. Using the bounds obtained
in the previous section we are able to prove that the rescaled
sequence $\{u^k\}$ has a convergent subsequence to a function
$U$. We establish the equation satisfied by the limit function $U$, as
well as the initial datum $U(x,0)$, depending on the conditions assumed
on $u_0$. In the case of nonintegrable initial data $u_0$, in order to
completely determine $U$, it is necessary to establish certain bounds on
the $L^1$ and $L^p$ norms of $u^k$. On the other hand, if $u_0$ is
integrable, we proceed in a different way, as can be seen at the end
of this section.

Let $u$ be a solution of
\begin{equation}\label{eq3}
\begin{cases}
u_t=Lu-u^p\qquad&\mbox{in }\,\R^N\times(0,\infty)\\
u(x,0) =  u_0(x)&\mbox{in }\,\R^N.
\end{cases}
\end{equation}
As defined in the introduction we denote by
$$u^k(x,t)=f(k)u(kx,k^2t),\,{\rm where }\;f(k)=\frac{k^N}{\int_{B_R}u_0}.$$

By Remark \ref{rmkulequL} we have that $f(t^{1/2})u(x,t)\le  C$. So
that,
by our assumption {\bf (F2)} on $f$, if $t\ge t_0$ there holds that
$$
u^k(x,t)\le C_{t_0}f(k\sqrt t)u(kx,k^2t)\le C_{t_0}.
$$

The function $u^k$ satisfies the following equation,
\begin{equation}\label{problem-rescaled}
u^k_t=k^2L_k u^k-F(k) (u^k)^p
\end{equation}
where $F(k)=f(k)^{1-p}k^2\to c_0\ge0$ as $k\to \infty$ by our assumptions, and the operator $L_k$ is defined by
\begin{equation}\label{Lk}
L_k v(x)=\big(J_k*v\big)(x)-v(x)=k^{N}\int J\big(k(x-y)\big)\big(v(y)-v(x)\big)\,dy.
\end{equation}

Our goal is to study the behavior of the sequence $\{u^k\}$. To do
so, we will decompose $u^k$ into an exponentially small part and
another one, $h^k$, depending on $W$, the smooth part of the
fundamental solution of the homogeneous linear problem.

Take $t_0>0$ and write $u(x,t)=e^{-(t-k^2t_0)}u(x,k^2t_0)+z(x,t)$. Then,
$$
z_t-Lz=e^{-(t-k^2t_0)}\big(J*u(\cdot,k^2t_0)\big)-u^p
$$
and $z(x,k^2t_0)=0$. After rescaling we have that
$u^k(x,t)=e^{-k^2(t-t_0)}u^k(x,t_0)+z^k(x,t)$. Observe that
$$e^{-k^2(t-t_0)}u^k(x,t_0)\le C_{t_0}e^{-k^2(t-t_0)}\to0,$$ as $k\to\infty$ uniformly in $t-t_0\ge c>0$,
so that the asymptotic behavior of $u^k$ for $k\to\infty$ is that of $z^k$.

By the variations of constants formula we get
$$
z(x,t)=\int_{k^2t_0}^t S(t-s)\Big[e^{-(s-k^2t_0)}\big(J*u(\cdot,k^2t_0)\big)(x)-u^p(x,s)\Big]\,ds
$$
where $S(t)$ is the semigroup associated to the homogeneous equation $u_t-Lu=0$.

Thus,
\begin{align*}
z(x,t)&= \int_{k^2t_0}^t e^{-(t-s)}\Big[e^{-(s-k^2t_0)}\big(J*u(\cdot,k^2t_0)\big)(x)-u^p(x,s)\Big]\,ds\\
&+\int_{k^2t_0}^t \int W(x-y,t-s)\Big[e^{-(s-k^2t_0)}\big(J*u(\cdot,k^2t_0)\big)(y)-u^p(y,s)\Big]\,dy\,ds\\
&=\int_{k^2t_0}^t e^{-(t-k^2t_0)}\big(J*u(\cdot,k^2t_0)\big)(x)\,ds-
\int_{k^2t_0}^t e^{-(t-s)}u^p(x,s)\,ds + h(x,t)\\
&=(t-k^2t_0)e^{-(t-k^2t_0)}\big(J*u(\cdot,k^2t_0)\big)(x)-
\int_{k^2t_0}^t e^{-(t-s)}u^p(x,s)\,ds + h(x,t)
\end{align*}
with
$$
h(x,t)=\int_{k^2t_0}^t \int W(x-y,t-s)\Big[e^{-(s-k^2t_0)}\big(J*u(\cdot,k^2t_0)\big)(y)-u^p(y,s)\Big]\,dy\,ds.
$$

Therefore,
\begin{align*}
z^k(x,t)&= k^2(t-t_0)e^{-k^2(t-t_0)}f(k)\big(J*u(\cdot,k^2t_0)\big)(kx)
-\int_{k^2t_0}^{k^2t}e^{-(k^2t-s)}f(k) u^p(kx,s)\,ds\\
&\ \ +h^k(x,t)\\
&=k^2(t-t_0)e^{-k^2(t-t_0)}f(k)\big(J*u(\cdot,k^2t_0)\big)(kx)-F(k)
\int_{t_0}^t e^{-k^2(t-s)}(u^k)^p(x,s)\,ds\\
&\ \ +h^k(x,t).
\end{align*}

There holds that,
\begin{align*}
&f(k) \big(J*u(\cdot,k^2t_0)\big)(kx)=\int J(kx-y)f(k) u(y,k^2t_0)\,dy\\
=&\big(J_k*u^k(\cdot,t_0)\big)(x).
\end{align*}

So that,
\begin{align*}
z^k(x,t)&= k^2(t-t_0)e^{-k^2(t-t_0)}\big(J_k*u^k(\cdot,t_0)\big)(x)-
F(k)
\int_{t_0}^t e^{-k^2(t-s)}(u^k)^p(x,s)\,ds\\
&\ \ +h^k(x,t).
\end{align*}

With similar computations we find that
\begin{align*}
h^k(x,t)&= k^2\int_{t_0}^t\int W_k(x-y,t-s) e^{-k^2(s-t_0)}\big(J_k*u^k(\cdot,t_0)\big)(y)\,dy\,ds\\
&\ \ -F(k)\int_{t_0}^t\int W_k(x-y,t-s) (u^k)^p(y,s)\,dy\,ds,
\end{align*}
where $W_k(x,t)=k^NW(kx,k^2t)$.

Let us now estimate the first term in the expansion of $z^k$. We have
$$
k^2(t-t_0)e^{-k^2(t-t_0)}\big(J_k*u^k(\cdot,t_0)\big)(x)\le C k^2(t-t_0)e^{-k^2(t-t_0)} C_{t_0}\to 0\quad\mbox{as }k\to\infty,
$$
uniformly in $t-t_0\ge c>0$.

For the second term, since $F(k)$ is bounded, we have the estimate
\begin{align*}
F(k)\int_{t_0}^t e^{-k^2(t-s)}(u^k)^p(x,s)\,ds&\le  C_{t_0}\int_{t_0}^{t}e^{-k^2(t-s)}\,ds\\
&\le C_{t_0} k^{-2}\to0\quad{as}\quad k\to\infty.
\end{align*}

Therefore, the asymptotic behavior as $k\to\infty$ of $u^k$ is that of $h^k$.

Since the functions $h^k$ are smooth, we can show that the family $\{h^k\}$ is
precompact in $C(\K)$ with $\K\subset\subset \R^N\times(t_0,\infty)$ by finding uniform Holder estimates.

We begin with estimates in space.
\begin{prop}\label{lipschitz} Let $t_0>0$ and $ T>2t_0$. There exists a constant $L>0$ such $|\nabla h^k(x,t)|\le L$ for $x\in \R^N$ if $t\in[2t_0,T]$.
\end{prop}
\begin{proof}
Recall that we have obtained estimates for the $L^1$ norm of $\nabla W(\cdot,t)$ (cf. \eqref{first-int-nablaW} and \eqref{second-int-nablaW}). Now we derive estimates for the $L^1$ norm of $\nabla W_k(x,t)=k k^N\nabla W(kx,k^2t)$.
 We have,
 $$\|\nabla W_k(\cdot,t)\|_{L^1(\R^N)}=k\|\nabla W(\cdot,k^2t)\|_{L^1(\R^N)}.$$

 Therefore,
 $$ \|\nabla W_k(\cdot,t)\|_{L^1(\R^N)}\le C k (k^2t)= C k^3 t.
 $$

 On the other hand,
$$ \|\nabla W_k(\cdot,t)\|_{L^1(\R^N)}\le C k (k^2t)^{-1/2}= C  t^{-1/2}.
 $$

 Differentiating the function $h^k$ we obtain,
 $$
 \begin{aligned}
 & |\nabla h^k(x,t)|\le k^2\int_{t_0}^t\int |\nabla W_k(x-y,t-s)|\, e^{-k^2(s-t_0)}\big(J_k*u^k(\cdot,t_0)\big)(y)\,dy\,ds\\
 &\hskip2cm +F(k)\int_{t_0}^t\int|\nabla W_k(x-y,t-s)|(u^k)^p(y,s)\,dy\,ds=
 I+II.
 \end{aligned}
 $$

 There holds,
 $$
 \begin{aligned}
 &I\le k^2 C_{t_0}\int_{t_0}^{3t/4} e^{-k^2(s-t_0)} (t-s)^{-1/2}\,ds+ k^2 C_{t_0}\int_{3t/4}^t e^{-k^2(s-t_0)}k^3(t-s)\,ds \\
 &\hskip2cm \le   C_{t_0}(t/4)^{-1/2}+C_{t_0} (t/4) k^3 e^{-k^2(\frac {3t}4-t_0)}\le C_{t_0,T}
 \end{aligned}
 $$

On the other hand,
\begin{align*}
& II\le C_{t_0}\int_{t_0}^t (t-s)^{-1/2}\,ds\le C_{t_0,T}.
\end{align*}
\end{proof}

Finally we prove one of our main results, namely, that the sequence
$\{h^k\}$ and therefore, the sequence $\{u^k\}$ is uniformly
convergent on compact sets.

\begin{teo}
There exists a subsequence that we still call $h^k$ which is uniformly convergent on every compact subset of $\R^N\times[2t_0,\infty)$.
\end{teo}
\begin{proof}
In order to prove the result, let us split $h^k$ into two terms.
\begin{align*}
h^k(x,t)&=k^2\int_{t_0}^t\int {W_k}(x-y,t-s) e^{-k^2(s-t_0)}\big(J_k*u^k(\cdot,t_0)\big)(y)\,dy\,ds\\
&\hskip2cm -F(k)\int_{t_0}^t\int {W_k}(x-y,t-s) (u^k)^p(y,s)\,dy\,ds\\
&=H_0^k(x,t)+ H^k(x,t).
\end{align*}
with
$$
H_0^k(x,t)=k^2\int_{t_0}^t\int {W_k}(x-y,t-s) e^{-k^2(s-t_0)}\big(J_k*u^k(\cdot,t_0)\big)(y)\,dy\,ds
$$
and
$$
H^k(x,t)=F(k)\int_{t_0}^t\int  W_k(x-y,t-s)(u^k)^p(y,s)\,dy\,ds.
$$

By the estimates of $\|W_t(\cdot,t)\|_{L^1(\R^N)}$ (cf. \eqref{integral-Wt}) we get
$$
\|{W_k}_t(\cdot,t)\|_{L^1}\le Ct^{-1}
$$
and
$$
\|{W_k}_t(\cdot,t)\|_{L^1}\le k^2 e^{-k^2t}+ C k^4 t.
$$

Therefore,
\begin{align*}
|{H_0^k}_t(x,t)|&\le k^2\int_{t_0}^t\int |{W_k}_t(x-y,t-s)| \, e^{-k^2(s-t_0)}\big(J_k*u^k(\cdot,t_0)\big)(y)\,dy\,ds\\
&\le Ck^2\int_{t_0}^{3t/4}(t-s)^{-1}C_{t_0}e^{-k^2(s-t_0)}\,ds+
Ck^2\int_{3t/4}^t k^2e^{-k^2(t-s)}C_{t_0}e^{-k^2(s-t_0)}\,ds\\
&+Ck^2\int_{3t/4}^t k^4(t-s)C_{t_0}e^{-k^2(s-t_0)}\,ds\\
&\le C_{t_0}(t/4)^{-1}+C_{t_0} (t/4)k^4e^{-k^2(t-t_0)}+C_{t_0}(t/4)k^4e^{-k^2(\frac{3t}4-t_0)}\\
&\le C_{t_0,T}.
\end{align*}

Therefore, the sequence $H_0^k$ has a subsequence that converges uniformly on compact subsets of $\R^N\times[2t_0,\infty)$.

In order to see that the same conclusion holds for the sequence $H^k$, let us define
$$
R^k(x,t):= F(k)\int_{t_0}^t\int  U_\a(x-y,t-s)(u^k)^p(y,s)\,dy\,ds,
$$
with $U_\a$ the fundamental solution of the heat equation with diffusivity $\a$. Then, for every $T>0$,
$$
\|R^k\|_{L^\infty(\R^N\times[t_0,T])}\le C_{t_0,T}
$$
and
$$
\|H^k-R^k\|_{L^\infty(\R^N\times[t_0,T])}\to 0\quad\mbox{as}\quad k\to \infty.
$$
In fact,
\begin{align*}
\|H^k(\cdot,t)-R^k(\cdot,t)\|_{L^\infty(\R^N)}\le F(k)\int_{t_0}^t\|W_k(\cdot,t-s)-U_\a(\cdot,t-s)\|_{q'}\|(u^k)^p(\cdot,s)\|_q\,ds.
\end{align*}

Recall that $f(k)\ge c k^\beta$ with $\beta=2/(p-1)$. Let us take $q>N/\beta p$. So that there holds,
$$
\|(u^k)^p(\cdot,s)\|_q\le \Big(\int\frac{dx}{\big(s^{1/2}+|x|\big)^{\beta pq}}\Big)^{1/q}\le C_{t_0,q} \quad\mbox{for}\quad s\ge t_0.
$$

On the other hand, since $U_\a(x,t)={U_\a}_k(x,t)$,
\begin{align*}
\|W_k(\cdot,t)-U_\a(\cdot,t)\|_{q'}&=\|W_k(\cdot,t)-{U_\a}_k(\cdot,t)\|_{q'}\\
&=k^{N/q}\|W(\cdot,k^2t)-U_\a(\cdot,k^2t)\|_{q'}\\
&\le k^{N/q}\,(k^2 t)^{-(N+1)/2q}=k^{-1/q} t^{-(N+1)/2q}\qquad\mbox{(cf. \cite{TW})}.
\end{align*}

Let us choose $q$ big enough so that we also have $q>(N+1)/2$. Then,
\begin{align*}
\|H^k(\cdot,t)-R^k(\cdot,t)\|_{L^\infty(\R^N)}&\le F(k) C_{q,t_0}k^{-1/q}\int_{t_0}^t(t-s)^{-(N+1)/2q}\,ds\\
&\le C_{q,t_0,T}\, k^{-\frac1q}\quad\mbox{if}\quad t_0\le t\le T.
\end{align*}

Hence,
$$
\|H^k-R^k\|_{L^\infty(\R^N\times[t_0,T])}\le  C_{q,t_0,T}\, k^{-\frac1q}\to 0\quad\mbox{as}\quad k\to\infty.
$$

Finally, observe that $R^k$ is a solution of the heat equation with diffusivity $\a$ and uniformly bounded right hand side. In fact,
$$
R^k_t(x,t)-\a \Delta R^k(x,t)=F(k)(u^k)^p(x,t)
$$
and
$$
0\le F(k)(u^k)^p(x,t)\le C_{t_0}\quad\mbox{if}\quad t\ge t_0.
$$

Therefore, the family $R^k$ is uniformly Holder continuous in $\R^N\times[2t_0,T]$ for every $T>2t_0$. We conclude that there exists a subsequence that is uniformly convergent on every compact subset of $\R^N\times[2t_0,\infty)$. And the same conclusion then holds for the family $H^k$.
\end{proof}

Now that we know that $\{u^k\}$ has a convergent subsequence, we
proceed in
identifying the limit function $U$. As a first step, and using the
assumptions on $F(k)$ in the introducion we establish the equation
satisfied by $U$.

\begin{prop}\label{U} Let $k_n\to\infty$ be such that $u^{k_n}\to U$ uniformly on compact sets of $\R^N\times(0,\infty)$. Then, $U$ is a solution to
\begin{equation}\label{eq-limit}
U_t-\a\Delta U=-c_0 U^p
\end{equation}
where $c_0=\lim_{k\to\infty}F(k)$.
\end{prop}
\begin{proof}
For simplicity we drop the subscript $n$. Let $\varphi\in C_0^\infty(\R^N\times(0,\infty))$ and $\mathcal K$ a compact set containing its support. Then,
\begin{eqnarray*}
\int_0^\infty\int_{\R^N} U(x,t)\big(\varphi_t+\a \Delta\varphi\big)(x,t)\,dx\,dt&=&\int_0^\infty\int_{\R^N}u^k(x,t)
\big(\varphi_t+k^2
L_k\varphi\big)(x,t)\,dx\,dt+\\
&&\hspace{-20,5em}+\int_0^\infty\int_{\R^N}(U-u^k)(x,t)\big(\varphi_t+\a \Delta\varphi\big)(x,t)\,dx\,dt-\int_0^\infty\int_{\R^N}u^k(x,t)\big(k^2L_k\varphi-
\a\Delta\varphi\big)(x,t)\,dx\,dt\\
&=&-\int_0^\infty\int_{\R^N}\big(u^k_t-k^2L_ku^k\big)(x,t)\varphi(x,t)\,dx\,dt+\\
&&\hspace{-17em}+\int_0^\infty\int_{\R^N}(U-u^k)(x,t)\big(\varphi_t+\a \Delta\varphi\big)(x,t)\,dx\,dt
-\int_0^\infty\int_{\R^N}u^k(x,t)\,O\big(k^{-3}\big)\chi_{\mathcal K}\,
dx\,dt\\
&=&F(k)\int_0^\infty\int_{\R^N} (u^k)^p(x,t)\varphi(x,t)\,dx\,dt\\
&&\hspace{-17em}+\int_0^\infty\int_{\R^N}(U-u^k)(x,t)\big(\varphi_t
+\a \Delta\varphi\big)(x,t)\,dx\,dt
-\int_0^\infty\int_{\R^N}u^k(x,t)\,O\big(k^{-3}\big)\chi_{\mathcal K}\,
dx\,dt\\
&\to &\ c_0\int_0^\infty\int_{\R^N} U^p(x,t)\varphi(x,t)\,dx\,dt\quad\mbox{as }k\to\infty.
\end{eqnarray*}

So that,
$$
\int_0^\infty\int_{\R^N} U(x,t)\big(\varphi_t+\a \Delta\varphi\big)(x,t)\,dx\,dt=\
c_0\int_0^\infty\int_{\R^N} U^p(x,t)\varphi(x,t)\,dx\,dt
$$
for every$\varphi\in C_0^\infty(\R^N\times(0,\infty))$. Thus,
$$
U_t-\a\Delta U=-c_0 U^p.$$
\end{proof}

Finally, in order to identify the initial datum $U(x,0)$, we need to take into
account the behavior of $u_0$ at infinity. Moreover, it is necessary to
have some control on the $L^1$ and $L^p$ norms of $u^k$  on sets of the form $B_R\times(0,\tau)$ for $R,\tau>0$.

\begin{prop}\label{determineUa} Let $u$ be a solution to \eqref{eq3} and $u^k$ its rescaling. Let $k_n\to\infty$ and assume $u^{k_n}\to U$ as $n\to\infty$. Assume that for every $R>0$ there exists $C_R$ such that
$$
\int_0^\tau \int_{B_R} u^k(x,t)\,dx\,dt\le C_R\tau$$
and either there exists $\gamma>0$ such that
$$F(k)\int_0^\tau \int_{B_R} (u^k)^p(x,t)\,dx\,dt\le C_R\tau^\gamma
$$
or else, for every $R,\tau>0$,
$$
\lim_{k\to\infty}F(k)\int_0^\tau \int_{B_R} (u^k)^p(x,t)\,dx\,dt=0.
$$

Assume further that
$$
u_0^k(x) \to \phi(x)\quad(k\to\infty)\quad\mbox{in the sense of distributions.}
$$

Then,
there holds that $U$ is the solution to
$$
\begin{aligned}
&U_t-\a \Delta U=-c_0 U^p\\
&U(x,0)=\phi(x)
\end{aligned}
$$
\end{prop}

\begin{proof}
We proceed as in the proof of Proposition \ref{U}. We drop the subscript $n$. Let $\varphi\in C_0^\infty(\R^N\times[0,\infty))$, $R>0$ such that $\varphi(x,t)=0$ if $|x|>R$ and let $\ep>0$. Let $\tau>0$, $k_0>0$ be such that for $k\ge k_0$,
$$
\begin{aligned}
&\int_0^\tau\int_{B_R} u^k(x,t)\,dx\,dt+F(k)\int_0^\tau\int_{B_R} (u^k)^p(x,t)\,dx\,dt<\ep\\
&\int_0^\tau\int_{B_R} U(x,t)\,dx\,dt+c_0\int_0^\tau\int_{B_R} U^p(x,t)\,dx\,dt<\ep
\end{aligned}
$$

Then,
$$
\begin{aligned}
&\Big|\int_0^\infty\int_{\R^N} U(x,t)\big(\varphi_t+\a \Delta\varphi\big)(x,t)\,dx\,dt-\int_0^\infty\int_{\R^N} c_0 U^p(x,t)\varphi(x,t)\,dx\,dt\\
+&\int_{\R^N} \phi(x)\varphi(x,0)\,dx\Big|\le C\Big|\int_0^\tau\int_{B_R} U(x,t)\,dx\,dt\Big|+C\Big|c_0\int_0^\tau\int_{B_R} U^p(x,t)\,dx\,dt\Big|\\
+&C\int_\tau^\infty\int_{B_R}|U-u^k|(x,t)\,dx\,dt+ \int_\tau^\infty\int_{B_R}\big|\,c_0 U^p-F(k)(u^k)^p\big|(x,t)\,dx\,dt\\
+&\int_\tau^\infty\int_{B_R}u^k(x,t) O(k^{-3})\,dx\,dt+C\Big|\int_0^\tau\int_{B_R} u^k(x,t)\,dx\,dt\Big|+C\Big|F(k)\int_0^\tau\int_{B_R} (u^k)^p(x,t)\,dx\,dt\Big|\\
+&\Big|\int_{\R^N}u_0^k(x)\varphi(x,0)\,dx-\int_{\R^N}\phi(x)\varphi(x,0)\,dx\Big|\le
2C\ep+C\int_\tau^\infty\int_{B_R}|U-u^k|(x,t)\,dx\,dt\\
+ &\int_\tau^\infty\int_{B_R}\big|\,c_0 U^p-F(k)(u^k)^p\big|(x,t)\,dx\,dt+
\int_\tau^\infty\int_{B_R}u^k(x,t) O(k^{-3})\,dx\,dt\\
+&\Big|\int_{\R^N}u_0^k(x)\varphi(x,0)\,dx-\int_{\R^N}\phi(x)\varphi(x,0)\,dx\Big|
\end{aligned}
$$

Therefore, taking $\limsup_{k\to\infty}$ we obtain,
$$\hspace{-5em}\Big|\int_0^\infty\int_{\R^N} U(x,t)\big(\varphi_t+\a \Delta\varphi\big)(x,t)\,dx\,dt-\int_0^\infty\int_{\R^N} c_0 U^p(x,t)\varphi(x,t)\,dx\,dt+$$
$$+\int_{\R^N} \phi(x)\varphi(x,0)\,dx\Big|\le 2C\ep.
$$

As $\ep$ is arbitrary, the proposition is proved.
\end{proof}

In order to prove the bounds assumed in the previous proposition, and
hence be able to completely determine $U$, we need to consider separate
cases, according to the behavior of $u_0$ at infinity. We begin with
the case where $u_0$ behaves as a power $-\alpha>-N$.

\begin{lema}\label{lemmaalpha} Let $u$ be the solution to \eqref{eq3}.
 Assume $(1+|x|)^{\alpha}u(x,t)\le B$ and $(1+t)^{\alpha/2}u(x,t)\le B$  with $0<\alpha< N$.
Then,  for every $R>0$ there exists $C_R$ such that,
for $k\ge \tau^{-1/2}$,
\begin{align*}
&\int_0^\tau \int_{B_R}u^k(x,t)\,dx\,dt\le C_R\tau\\
F(k)&\int_0^\tau\int_{B_R} (u^k)^p(x,t)\,dx\,dt\le C_R k^{-\alpha(p-1)+2}\tau\le C_R\tau\hskip2.3cm\;\mbox{if }\ N-\alpha p>0\\
F(k)&\int_0^\tau\int_{B_R} (u^k)^p(x,t)\,dx\,dt\le C_R k^{-\alpha(p-1)+2}\tau|\log\tau|\le C_R\tau|\log\tau|\;\mbox{if }\ N-\alpha p=0\\
F(k)&\int_0^\tau\int_{B_R} (u^k)^p(x,t)\,dx\,dt\le C k^{-\alpha(p-1)+2} \tau^{\frac{N-\alpha p+2}2}\le C \tau^{\frac{N-\alpha p+2}2}
\hskip.4cm\;\mbox{if }\ 0>N-\alpha p>-2\\
F(k)&\int_0^\tau\int_{B_R} (u^k)^p(x,t)\,dx\,dt\le C k^{\alpha-N}\le C_R\tau^{\frac{N-\alpha}2}\hskip2.9cm\;\mbox{if }\ N-\alpha p<-2\\
F(k)&\int_0^\tau\int_{B_R} (u^k)^p(x,t)\,dx\,dt\le C k^{\alpha-N}\log(1+k^2\tau)\le C \tau^{(N-\alpha)/2}\hskip.4cm\; \mbox{if }\ N-\alpha p=-2.
\end{align*}
\end{lema}
\begin{proof}
We begin with the estimate of the integral of $u^k$.
\begin{eqnarray}\label{ualpha}
\int_0^\tau\int_{B_R} u^k(x,t)\,dx\,dt&=&f(k) k^{-N-2}\int_0^{k^2\tau}\int_{B_{Rk}}u(x,t)\,dx\,dt\\
\nonumber&\le& Ck^{\alpha-N-2}\int_0^{k^2\tau}\int_{B_{Rk}}\frac1{(1+|x|)^\alpha}\,dx\,dt\\
\nonumber&\le& C k^{\alpha-N-2}k^2\tau (Rk)^{N-\alpha}\le C_R\tau.
\end{eqnarray}
Let us now estimate the integral of $(u^k)^p$.
There holds,
\begin{eqnarray}\label{up}
F(k)\int_0^\tau\int_{B_R} (u^k)^p(x,t)\,dx\,dt&=& F(k) f(k)^pk^{-N-2}\int_0^{k^2\tau}\int_{B_{Rk}}u^p(x,t)\,dx\,dt\\
\nonumber&=&f(k)k^{-N}\int_0^{k^2\tau}\int_{B_{Rk}}u^p(x,t)\,dx\,dt.
\end{eqnarray}

\medskip

We consider several cases.

\medskip

\noindent{\tt Case 1}: $N-\alpha p>0$

 We have,
\begin{eqnarray*}
F(k)\int_0^\tau\int_{B_R} (u^k)^p(x,t)\,dx\,dt& \le& C  F(k) k^{\alpha p-N-2}\int_0^{k^2\tau}\int_{B_{Rk}}\frac1{(1+|x|)^{\alpha p}}\,dx\,dt\\
&\le& C F(k) k^{\alpha p-N-2}k^2\tau(Rk)^{N-\alpha p}\\
&=&C_R k^{-\alpha(p-1)+2}\tau\le C_R\tau.
\end{eqnarray*}

\noindent{\tt Case 2}: $N-\alpha p<0$

 First,
\begin{eqnarray}\label{up-space}
\int_{B_{Rk}} u^p(x,t)\,dx&\le& C\int_{|x|\le\sqrt t} (1+t)^{-\alpha p/2}\,dx+C
\int_{\sqrt t<|x|<Rk}\frac1{(1+|x|)^{\alpha p}}\,dx\\
\nonumber&\le& C (1+t)^{\frac{N-\alpha p}2}.
\end{eqnarray}

\medskip

Assume $k\ge \tau^{-1/2}$. We consider 3 subcases.

\medskip

\begin{enumerate}
\item[(i)] $-2<N-\alpha p<0$.
\begin{eqnarray*}
F(k)\int_0^\tau\int_{B_R} (u^k)^p(x,t)\,dx\,dt&\le &C k^{\alpha -N}\int_0^{k^2\tau}(1+t)^\frac{N-\alpha p}2\,dt\\
&\le &C k^{\alpha -N}(k^2\tau)^{\frac {N-\alpha p+2}2}\\
&=&C k^{-\alpha(p-1)+2}\tau^{\frac {N-\alpha p+2}2}\\
 &\le &C\tau^{\frac {N-\alpha p+2}2}.
\end{eqnarray*}

\item[(ii)] $N-\alpha p=-2$.
\begin{eqnarray*}
F(k)\int_0^\tau\int_{B_R} (u^k)^p(x,t)\,dx\,dt&\le& C k^{\alpha
  -N}\int_0^{k^2\tau}(1+t)^{-1}\,dt\\
&=&C k^{\alpha -N}\log(1+k^2\tau)\\
&=&C(k^2\tau)^{(\alpha-N)/2}\log(1+k^2\tau)\tau^{(N-\alpha)/2}\\
&\le& C \tau^{(N-\alpha)/2}.
\end{eqnarray*}

\item[(iii)] $N-\alpha p<-2$.
\begin{eqnarray*}
F(k)\int_0^\tau\int_{B_R} (u^k)^p(x,t)\,dx\,dt&\le& C k^{\alpha -N}\int_0^{k^2\tau}(1+t)^{(N-\alpha p)/2}\,dt\\
&\le& C k^{\alpha -N}\le C\tau^{(N-\alpha)/2}.
\end{eqnarray*}

\end{enumerate}

\medskip

\noindent{\tt Case 3}: $N-\alpha p=0$

Instead of \eqref{up-space} we have,
$$
\int_{B_{Rk}}u^p(x,t)\,dx\le C\Big(1+\log\frac{Rk}{\sqrt t}\Big).
$$
Thus,
\begin{eqnarray*}
F(k)\int_0^\tau\int_{B_R} (u^k)^p(x,t)\,dx\,dt&\le& C k^{\alpha-N}\int_0^{k^2\tau} \Big(1+\log\frac{Rk}{\sqrt t}\Big)\,dt\\
&=&C k^{\alpha-N+2}\int_0^{\tau} \Big(1+\log\frac{R}{\sqrt
  t}\Big)\,dt\\
&\le& C_R k^{-\alpha(p-1)+2}\tau|\log\tau|\le C_R \tau|\log\tau|.
\end{eqnarray*}\end{proof}

\begin{coro} Assume $|x|^\alpha u_0(x)\to A>0$ as $|x|\to\infty$ with $0<\alpha<N$. Let $u$ be the solution to \eqref{eq3}. Then, $u$ satisfies the conclusions of Lemma \ref{lemmaalpha}
\end{coro}
\begin{proof} By Remark \ref{rmkulequL}, the estimates of
Proposition \ref{barrier-uL}
 hold for $u$. In this case, since $f(k)\sim k^{\alpha}$ we take
it to be equal and hence,
$$(1+|x|)^\alpha u(x,t)\le C\quad {\rm and }\quad  (1+t)^{\alpha/2} u(x,t)\le C.$$

Thus, $u$ satisfies the assumptions of Lemma \ref{lemmaalpha}.
\end{proof}

\medskip

Using Lemma \ref{lemmaalpha}, we are able to prove with almost no computations,
the desired estimates in the other examples considered in the introduction.  We use the ideas of Kamin-Ughi in \cite{KU}.

Let us introduce some notation. For $\mu>0$, let
\begin{equation}\label{ukmu}
u^{k,\mu}(x,t)=k^\mu u(kx,k^2t).
\end{equation}

Then,
\begin{equation}\label{upmu}
F(k)\int_0^\tau\int_{B_R}(u^k)^p(x,t)\,dx\,dt=f(k)k^{-\mu}k^{-\mu(p-1)+2}\int_0^\tau\int_{B_R} (u^{k,\mu})^p(x,t)\,dx\,dt.
\end{equation}

We begin with the case where $u_0$ behaves as $|x|^{-N}$ at
infinity.
\begin{lema}Assume $|x|^N u_0(x)\to A>0$ as $|x|\to\infty$. Let  $p>1+2/N$. Then,  for every $R>0$ there exists $C_R$  such that,
\begin{align}
&\int_0^\tau\int_{B_R} u^k(x,t)\,dx\,dt\le C_R\tau\\
 F(k)&\int_0^\tau\int_{B_R} (u^k)^p(x,t)\,dx\,dt\le \frac C{\log k}.\label{N}
\end{align}
\end{lema}
\begin{proof}
Recall that in this case $f(k)\sim \frac{k^N}{\log k}$ for $k\ge2$. Let $0<\mu<N$ to be chosen later. Then, for $|x|\ge1$, $t\ge1$,
$$
|x|^\mu u(x,t)\le C\frac{|x|^N}{\log(1+|x|)} u(x,t)\le B,\qquad t^{\mu/2} u(x,t)\le C\frac{t^{N/2}}{\log(1+t)}u(x,t)\le B.
$$
So that, we can apply the results of Lemma \ref{lemmaalpha} to $u^{k,\mu}$. Let us choose $\mu$ so close to $N$ so that $\mu p>N+2(>\mu+2)$.
By \eqref{upmu} and  Lemma \ref{lemmaalpha}, Case 2, (iii), there holds,
\begin{align*}
F(k)\int_0^\tau\int_{B_R}(u^k)^p(x,t)\,dx\,dt&\le C f(k)k^{-\mu}k^{\mu-N}
= \frac{C}{\log k}.
\end{align*}

On the other hand,

$$\begin{aligned}
\int_0^\tau\int_{B_R}u^k(x,t)\,dx\,dt&=\frac{k^{-2}}{\log k}
\int_0^{k^2\tau}\int_{B_{Rk}}
u(x,t)\,dx\,dt.
\end{aligned}$$

Now,
$$
\begin{aligned}
\int_{B_{Rk}}u(x,t)\,dx\le&e^{-t}\int_{B_{Rk}}u_0(x)\,dx+\int\int_{B_{Rk}} W(y,t)u_0(x-y)\,dx\,dy\\
=&e^{-t}\int_{B_{Rk}}u_0(x)\,dx+\int_{|y|<2Rk}\int_{B_{Rk}} W(y,t)u_0(x-y)\,dx\,dy\\
+&\int_{|y|>2Rk}\int_{B_{Rk}} W(y,t)u_0(x-y)\,dx\,dy
\le C e^{-t}\log(Rk)+I+II.
\end{aligned}
$$

There holds,
$$
I\le \int_{|y|<2Rk}W(y,t)\int_{|x|<3Rk}u_0(x)\,dx\,dy\le C\log(3Rk).
$$

On the other hand,
$$
\begin{aligned}
II\le &\int_{|y|>2Rk}W(y,t)\int_{|x|<Rk}\frac{C}{|y|^N}\,dx\,dy\\
\le & C(Rk)^N\int_{|y|>2Rk}\frac t{|y|^{2N+2}}\,dy\le C_K k^{-2}\,t.
\end{aligned}
$$

Thus,
$$
\begin{aligned}
\int_0^\tau \int_{B_R}u^k(x,t)\,dx\,dt\le & C \frac{\log(Rk)}{\log k}k^{-2}\int_0^{k^2\tau}e^{-t}\,dt\\
+& C\frac{\log(3Rk)}{\log k} k^{-2}k^2\tau
+ C_K k^{-4}\int_0^{k^2\tau}t\,dt\\
\le & C_R \tau\qquad\mbox{if }\tau<1.
\end{aligned}
$$
\end{proof}

With similar computations we can prove
\begin{lema} Assume $|x|^N(\log |x|) u_0(x)\to A>0$ as $|x|\to\infty$. Let $p>1+\frac2N$. Then,  for every $R>0$ there exists $C_R$ such that
\begin{align}\label{N-log}
&\int_0^\tau\int_{B_R}u^k(x,t)\,dx\,dt\le C_K\tau\\
F(k)&\int_0^\tau\int_{B_R}(u^k)^p(x,t)\,dx\,dt\le \frac C{\log\log k}.
\end{align}
\end{lema}

Also,
\begin{lema} Assume $\frac{|x|^N}{\log |x|} u_0(x)\to A>0$ as $|x|\to\infty$. Let $p>1+\frac2N$. Then,  for every $R>0$ there exists $C_R$ such that
\begin{align}\label{N/log}
&\int_0^\tau\int_{B_R}u^k(x,t)\,dx\,dt\le C_K\tau\\
F(k)&\int_0^\tau\int_{B_R}(u^k)^p(x,t)\,dx\,dt\le \frac C{\log^2 k}.
\end{align}
\end{lema}

In a similar way we obtain,
\begin{lema} Assume $ \frac{|x|^\alpha}{\log|x|} u_0(x)\to A>0$
as $|x|\to\infty$ with $0<\alpha<N$. Let $p>1+\frac2\alpha$. Then, for every $R>0$ there exists $C_R$ such that, for $k\ge \tau^{-1/2}$,
\begin{align*}
&\int_0^\tau \int_{B_R}u^k(x,t)\,dx\,dt\le C_R\tau\\
F(k)&\int_0^\tau\int_{B_R} (u^k)^p(x,t)\,dx\,dt\le C_R \tau\quad\mbox{if }\ N-\alpha p\ge0\\
F(k)&\int_0^\tau\int_{B_R} (u^k)^p(x,t)\,dx\,dt\le C \tau^{\frac{N-\mu p+2}2}
\quad\mbox{for a certain }\mu\in(0,\alpha)\;\mbox{if }\ 0>N-\alpha p\ge-2\\
F(k)&\int_0^\tau\int_{B_R} (u^k)^p(x,t)\,dx\,dt\le C_R\tau^{\frac{N-\alpha}2}\quad \mbox{if }\ N-\alpha p<-2.
\end{align*}
\end{lema}
\begin{proof} In this case $f(k)=\frac{k^\alpha}{\log(1+k)}$. We will chose $\mu\in(0,\alpha)$ according to the relative sizes of $\alpha$ and $p$. For any such $\mu$ there holds that $|x|^\mu u(x,t)\le C \frac{|x|^\alpha}{\log (1+|x|)} u(x,t)\le B$ if $|x|\ge 1$, $t^{\mu/2}u(x,t)\le C\frac{t^{\alpha/2}}{\log(1+t)}u(x,t)$ if $t\ge1$. So that, we can apply the results of Lemma \ref{lemmaalpha} to $u^{k,\mu}$.

Recall that $\alpha p>\alpha+2$. We consider several cases.

\noindent{\tt Case 1}: $N-\alpha p\ge0$

 Chose $\mu\in (0,\alpha)$ so that $\alpha+2<\mu p<N$. Then, by \eqref{upmu} and Lemma \ref{lemmaalpha}, Case (i),
    \begin{align*}
    F(k)\int_0^\tau\int_{B_R} (u^k)^p(x,t)\,dx\,dt\le C_R f(k) k^{-\mu}k^{-\mu(p-1)+2} \tau=C_R\frac{k^{\alpha+2-\mu p}}{\log k}\,\tau\le C_R\tau.
    \end{align*}

    \noindent{\tt Case 2}: $N-\alpha p<0$
    \medskip

    \begin{enumerate}
    \item[(i)] $-2\le N-\alpha p<0$. Chose $\mu\in (0,\alpha)$ so that
    $-2<N-\mu p<0$ and $\mu p>\alpha+2$. Then, by \eqref{upmu} and Lemma \ref{lemmaalpha}, Case 2, (ii),
    \begin{align*}
    F(k)\int_0^\tau\int_{B_R} (u^k)^p(x,t)\,dx\,dt&\le C_R f(k) k^{-\mu}k^{-\mu(p-1)+2}\,\tau^{\frac{N-\mu p+2}2}\\
    &=C_R\frac{k^{\alpha+2-\mu p}}{\log k}\tau^{\frac{N-\mu p+2}2}\le C_R \tau^{\frac{N-\mu p+2}2}.
    \end{align*}

    \item[(ii)] $N-\alpha p<-2$. Choose $\mu\in(0,\alpha)$ so that $N-\mu p<-2$. Then, by \eqref{upmu} and Lemma \ref{lemmaalpha}, Case 2, (iii),
      \begin{align*}
    F(k)\int_0^\tau\int_{B_R} (u^k)^p(x,t)\,dx\,dt&\le C f(k) k^{-\mu }k^{\mu-N}=C\frac{k^{\alpha-N}}{\log k}\le C\tau^{\frac{N-\alpha}2}.
    \end{align*}

    \end{enumerate}
    On the other hand, computations similar to those in \eqref{ualpha} give
    $$
    \int_0^\tau \int_{B_R}u^k(x,t)\,dx\,dt\le C_R\tau.
    $$
     This time we use that
    $$
    \int_{B_{Rk}}\frac{dx}{(2+|x|)^\alpha\log(2+|x|)}\sim (Rk)^N (f(Rk))^{-1} \sim{(Rk)^{N-\alpha}}{\log(Rk)}\qquad\mbox{as }k\to\infty.
    $$
    \end{proof}

    With similar computations we can prove,
    \begin{lema}\label{alpha-log-super} Assume $ {|x|^\alpha}(\log|x|)u_0(x)\to A>0$
as $|x|\to\infty$ with $0<\alpha<N$. Let $p>1+\frac2\alpha$. Then, for every $R>0$ there exists $C_R$ such that, for $k\ge\tau^{-1/2}$,
\begin{align*}
&\int_0^\tau \int_{B_R}u^k(x,t)\,dx\,dt\le C_R\tau\\
F(k)&\int_0^\tau\int_{B_R} (u^k)^p(x,t)\,dx\,dt\le C_R \tau\quad\mbox{if }\ N-\alpha p\ge0\\
F(k)&\int_0^\tau\int_{B_R} (u^k)^p(x,t)\,dx\,dt\le C \tau^{\frac{N-\mu p+2}2}
\quad\mbox{for a certain }\mu\in(0,\alpha)\;\mbox{if }\ 0>N-\alpha p\ge-2\\
F(k)&\int_0^\tau\int_{B_R} (u^k)^p(x,t)\,dx\,dt\le C_Rk^{\alpha-N}\log k\le C_R\tau^{\frac{N-\alpha}4}\quad \mbox{if }\ N-\alpha p<-2.
\end{align*}
\end{lema}

In order to consider the critical case $p=1+2/\alpha$ we need to perform different computations. There holds,
\begin{lema} Assume $ {|x|^\alpha}(\log|x|)u_0(x)\to A>0$
as $|x|\to\infty$ with $0<\alpha<N$. Let $p=1+\frac2\alpha$. Then, for every $\tau,R>0$ there exists $C_R$ such that,
\begin{align*}
&\int_0^\tau \int_{B_R}u^k(x,t)\,dx\,dt\le C_R\tau
\end{align*}
and
$$
\lim_{k\to\infty}F(k)\int_0^\tau\int_{B_R} (u^k)^p(x,t)\,dx\,dt=0.
$$
\end{lema}
\begin{proof} The $L^1$ estimate follows as in the previous lemmas. For the $L^p$ estimate we consider several cases.

\noindent{\tt Case 1:} $N-\alpha p>0$. There holds,
$$
\int_{B_{Rk}}u^p(x,t)\,dx\le C\int_{B_{Rk}}\frac{dx}{(2+|x|)^{\alpha p}\log^p(2+|x|)}\le C\int_0^{Rk}\frac{(2+r)^{N-\alpha p-1}}{\log^p(2+r)}\,dr
$$
where the last integral goes to infinity as $k$ goes to infinity. Moreover,
\begin{eqnarray*}
\lim_{k\to\infty}\frac{\int_0^{Rk}\frac{(2+r)^{N-\alpha p-1}}{\log^p(2+r)}\,dr}{\frac{k^{N-\alpha p}}{\log^p k}}&=&\lim_{k\to\infty}\frac{\frac{R(2+Rk)^{N-\alpha p-1}}{\log^p(2+Rk)}}{
\frac {(N-\alpha p)k^{N-\alpha p-1}}{\log^p k}-\frac{p\,k^{N-\alpha p-1}}{\log^{p+1} k}}\\
&=&\lim_{k\to\infty}
\frac{(2+Rk)^{N-\alpha p-1}}{k^{N-\alpha p-1}}\frac{\log^p
  k}{\log^p(2+Rk)}\frac R{N-\alpha p-\frac p{\log k}}\\
&= &C_{N,R,\alpha,p}<\infty.
\end{eqnarray*}
Therefore, for $k$ large,
\begin{eqnarray*}
F(k)\int_0^\tau\int_{B_R}(u^k)^p(x,t)\,dx\,dt&\le &C\tau k^{\alpha+2-N}\log k\int_0^{Rk}\frac{(2+r)^{N-\alpha p-1}}{\log^p(2+r)}\,dr\\
&\le& C\tau k^{\alpha+2-N}\log k\frac{k^{N-\alpha p}}{\log^p k}\\
&=&C\tau
\frac{k^{-\alpha(p-1)+2}}{\log^{p-1}k}=C\tau\frac1{\log^{p-1}k}\to 0\quad\mbox{as }k\to\infty.
\end{eqnarray*}

\medskip

\noindent{\tt Case 2:} $N-\alpha p< 0$.
There holds,

\begin{eqnarray*}
\int_{B_{Rk}}u^p(x,t)\,dx&\le& C\int_{|x|\le \sqrt
  t}\frac{dx}{(2+t^{1/2})^{\alpha p}\log^p(2+t^{1/2})}+\\
& & \hspace{3em}+C\int_{\sqrt t\le |x|\le Rk}\frac{dx}{(2+|x|)^{\alpha p}\log^p(2+|x|)}\\
&\le& C\frac{(2+t^{1/2})^{N-\alpha p}}{\log^p(2+t^{1/2})}
+C\int_{\sqrt t}^{Rk}\frac{(2+r)^{N-\alpha p-1}}{\log^p(2+r)}\,dr\\
&\le&
C\frac{(2+t^{1/2})^{N-\alpha p}}{\log^p(2+t^{1/2})}
+C\frac1{\log^p(2+t^{1/2})}\int_{\sqrt t}^{Rk}(2+r)^{N-\alpha
  p-1}\,{dr}\\
&\le& C \frac{(2+t^{1/2})^{N-\alpha p}}{\log^{p}(2+t^{1/2})}
\end{eqnarray*}

Since $\alpha p= \alpha+2$, there holds that $N-\alpha p=N-\alpha-2>-2$. We have,
\begin{align*}
F(k)\int_0^\tau\int_{B_R}(u^k)^p(x,t)\,dx\,dt&\le C k^{\alpha-N}\log k\int_0^{k^2 \tau}\frac{(2+t^{1/2})^{N-\alpha p}}{\log^{p}(2+t^{1/2})}\,dt,
\end{align*}
and this last integral goes to infinity as $k$ goes to infinity. Thus,
\begin{align*}
\lim_{k\to\infty}F(k)&\int_0^\tau\int_{B_R}(u^k)^p(x,t)\,dx\,dt\le \lim_{k\to\infty}
\frac{C\int_0^{k^2 \tau}\frac{(2+t^{1/2})^{N-\alpha p}}{\log^{p}(2+t^{1/2})}\,dt}{\frac{k^{N-\alpha}}{\log k}}\\
=&\lim_{k\to\infty}\frac{\frac{C2k\tau(2+k\tau^{1/2})^{N-\alpha p}}{\log^p(2+k\tau^{1/2})}}{\frac{(N-\alpha)k^{N-\alpha-1}}{\log k}-\frac{k^{N-\alpha-1}}{\log^2 k}}\\
=&\lim_{k\to\infty}
\frac{C2k\tau(2+k\tau^{1/2})^{N-\alpha-2}}{k^{N-\alpha-1}}\frac{\log k}{\log^p(2+k\tau^{1/2})}\frac1{N-\alpha-\frac1{\log k}}=0
\end{align*}

\medskip

\noindent{\tt Case 3:} $N-\alpha p=0$. We begin as in Case 2.

\begin{eqnarray*}
\int_{B_{Rk}}u^p(x,t)\,dx&\le& C\int_{|x|\le \sqrt t}\frac{dx}{(2+t^{1/2})^{\alpha p}\log^p(2+t^{1/2})}+\\
&&\hspace{3em}+C\int_{\sqrt t\le |x|\le Rk}\frac{dx}{(2+|x|)^{\alpha p}\log^p(2+|x|)}\\
&\le& C\frac{1}{\log^p(2+t^{1/2})}
+C\int_{\sqrt t}^{Rk}\frac{(2+r)^{-1}}{\log^p(2+r)}\,dr\\
&\le&
C\frac{1}{\log^{p}(2+t^{1/2})}\Big(1+\log\frac{2+Rk}{2+t^{1/2}}\Big)
\end{eqnarray*}

Thus,
\begin{align*}
F(k)\int_0^\tau\int_{B_R}(u^k)^p(x,t)\,dx\,dt&\le C k^{\alpha-N}\log k\int_0^{k^2 \tau}\frac{1}{\log^{p}(2+t^{1/2})}\Big(1+\log\frac{2+Rk}{2+t^{1/2}}\Big)\,dt,
\end{align*}
and this last integral goes to infinity as $k$ goes to infinity. Therefore, since $N-\alpha-1=1$,
\begin{align*}
\lim_{k\to\infty}F(k)&\int_0^\tau\int_{B_R}(u^k)^p(x,t)\,dx\,dt\le \lim_{k\to\infty}
\frac{C\int_0^{k^2 \tau}\frac{1}{\log^{p}(2+t^{1/2})}\big(1+\log\frac{2+Rk}{2+t^{1/2}}\big)\,dt}{\frac{k^{N-\alpha}}{\log k}}\\
=&\lim_{k\to\infty}\frac{\frac{C2k\tau}{\log^{p}(2+k\tau^{1/2})}\big(1+\log\frac{2+Rk}{2+k\tau^{1/2}}\big)}{\frac{(N-\alpha)k^{N-\alpha-1}}{\log k}-\frac{k^{N-\alpha-1}}{\log^2 k}}\\
=&\lim_{k\to\infty}
{2C\tau}\frac{\log k}{\log^{p}(2+k\tau^{1/2})}\Big(1+\log\frac{2+Rk}{2+k\tau^{1/2}}\Big)\frac1{N-\alpha-\frac1{\log k}}=0
\end{align*}
\end{proof}

\bigskip

If $u_0\in L^1(\R^N)$, the limit $U$ of $u^k$ does not satisfy that $U(x,0)=\lim u_0^k=M_0\delta$ with $M_0=\int u_0$. Therefore, estimates as the ones we have just stated do not hold. In order to establish our result we need the  following lemma.
\begin{lema}\label{compact} Let $u_0\in L^\infty(\R^N)$. Assume
  further that
$|x|^{N+2}u_0(x)\le B$. Let $u$ be the solution to \eqref{eq3} and $u^k$ its rescaling. Assume $u^{k_n}\to U$  with $k_n\to\infty$ as $n\to\infty$. Then, there exists a constant $C>0$ such that
$$
U(x,t)\le C\frac t{|x|^{N+2}}.
$$

In particular, for every $\mu>0$,
$$
\lim_{t\to0}U(x,t)=0 \qquad\mbox{uniformly in }\ |x|
\ge \mu.
$$
\end{lema}
\begin{proof} We proceed as in the proof of Theorem \ref{teo-barrier-W}. There holds that $u(x,t)\le e^{-t}u_0(x)+z(x,t)$ with $z$ the solution to
\begin{equation}\label{z}
\left\{\begin{aligned}
&z_t-Lz=e^{-t}(J*u_0)(x)\\
&z(x,0)=0
\end{aligned}\right.
\end{equation}

By the assumption on the growth of $u_0$ at infinity there holds that $v(x,t)=C\frac t{|x|^{N+2}}$ is a supersolution to \eqref{z} in $\frac t{|x|^2}\le K$, $|x|\ge 2$. On the other hand, $z(x,t)\le v(x,t)$ in the complement of this set since $z(x,t)\le Ct$ and $t^{N/2}z(x,t)\le C$ in this case ($u_0\in L^1(\R^N)$).

Therefore, $u(x,t)\le e^{-t} u_0(x)+ C\frac t{|x|^{N+2}}$. Rescaling, and recalling that in the present case $f(k)=k^N$ we have that  $u^k(x,t) \le e^{-k^2t}u_0^k(x)+ C\frac t{|x|^{N+2}}$. Passing to the limit as $k\to\infty$ with $x\neq0$ we get $U(x,t)\le C\frac t{|x|^{N+2}}$. And the result follows.
\end{proof}

\bigskip

\begin{prop} Let $u_0\in L^1\cap L^\infty$. Let $u$ be the solution to \eqref{eq3} and $u^k$ its rescaling. Assume for some sequence $k_n\to\infty$ there holds that $u^{k_n}\to U$ uniformly on compact sets of $\R^N\times(0,\infty)$.

Assume $p>1+2/N$. Then, $U(x,t)=M U_\a(x,t)$ where $U_\a$ is  the fundamental solution of the heat equation with diffusivity $\a$ and $M=\int u_0(x)\,dx-\int_0^\infty\int u^p(x,t)\,dx\,dt$.

When $p=1+2/N$ there holds that $U\equiv0$.
\end{prop}
\begin{proof} As before we drop the subscript $n$. We already know that $U$ is a solution to
$$
U_t-\a\Delta U=-c_0 U^p.
$$

Assume for simplicity that $u_0$ has compact support. Then, by Lemma \ref{compact} there exists $C>0$ such that $U(x,t)\le C\frac t{|x|^{N+2}}$.
Let $M(t)=\int U(x,t)\,dx$. We will see that $M(t)\equiv M$ is constant, and  $U(x,t)\to M \delta$ as $t\to0$ in the sense of distributions where $\delta$ is the Dirac delta. In fact, assume we already proved that $M(t)$ is constant $M$. Let $\varphi\in C_0^\infty(\R^N)$. Then, given $\ep>0$, if $\mu$ is small we get,
\begin{align*}
&\big|\int U(x,t)\varphi(x)\,dx- M\varphi(0)\big|=\big|\int
U(x,t)\big[\varphi(x)-\varphi(0)\big]\,dx\big|\\
&\hskip3cm\le \int_{|x|<\mu}U(x,t)|\varphi(x)-\varphi(0)|\,dx+
C\int_{|x|>\mu}\frac t{|x|^{N+2}}|\varphi(x)-\varphi(0)|\,dx\\
&\hskip3cm\le \ep\int U(x,t)\,dx+ 2C\|\varphi\|_\infty \int_{|x|>\mu}\frac
t{|x|^{N+2}}\,dx\\
&\hskip3cm\le M\ep +\bar C_\mu \,t<(M+1)\ep\qquad\mbox{if }t\mbox{ is small}.
\end{align*}

Now, integrating the equation satisfied by $u^k$ and using the
symmetry of $J$ we obtain
\begin{eqnarray}\label{masa}
\int u^k(x,t)\,dx&=&\int u_0^k(x)\,dx-F(k)\int_0^t\int (u^k)^p(x,s)\,dx\,ds\\
\nonumber &=&\int u_0(x)-\int_0^{k^2 t}\int u^p(x,s)\,dx\, ds.
\end{eqnarray}
Moreover
\begin{equation}\label{masaK}
\int_{B_K}u^k(x,t)\,dx\to\int_{B_K} U(x,t)\,dx\quad\mbox{for every
}K>0
\end{equation}
and
\begin{equation}\label{masaKc}
\begin{aligned}
&\int_{|x|>K} u^k(x,t)\,dx\le e^{-k^2t}\int_{|x|>Kk}u_0(x)\,dx+Ct
\int_{|x|>K}\frac{dx}{|x|^{N+2}} <\ep\\
& \int_{|x|>K} U(x,t)\,dx\le Ct
\int_{|x|>K}\frac{dx}{|x|^{N+2}} <\ep
\end{aligned}
\end{equation}
if K is large.

Hence, we have that
$$M(t)=\int U(x,t)\,dx=\int u_0(x)\,dx-\int_0^\infty\int u^p(x,t)\,dx\,dt=M.
$$

When $p>1+2/N$ there holds that $c_0=0$, so that $U$ is $M$ times the fundamental  solution of the heat equation with diffusivity $\a$. On the other hand, if
$p=1+2/N$ there holds that $c_0=1$, so that
$$
\int U(x,t)\,dx=\int U(x,\tau)\,dx-\int_\tau^t\int U^p(x,s)\,dx\,ds.
$$

Therefore, $M(t)$ cannot be constant unless $U\equiv0$.
And the proposition is proved when $u_0$ has compact support.

\bigskip

Let now $u_0\in L^\infty\cap L^1$ be arbitrary.  We have to prove that $u^k$ converges to  $U=M U_\a$ uniformly on compact subsets of $\R^N\times(0,\infty)$ with $M=\int u_0(y)\,dy-\int_0^\infty\int u^p(y,t)\,dy\,dt$.

Let $u_0^n=u_0\chi_{|x|<n}$ and $u_n$ the solution to \eqref{eq3} with initial datum $u_0^n$. Then, $u_n\le u$. Let $v_n=u-u_n$. There holds that $v_n$ is a nonnegative solution to
$$
{v_n}_t-L v_n=-u^p+u_n^p\le 0.
$$

Therefore,
$$
0\le v_n(x,t)\le e^{-t}\{u_0(x)-u_0^n(x)\}+\int W(x-y,t)\{u_0(y)-u_0^n(y)\}\,dy.
$$

Thus, if $t\ge\tau>0$,
$$
\begin{aligned}
0\le v_n^k(x,t)\le &k^Ne^{-k^2 t} \{u_0(kx)-u_0^n(kx)\}+\int  W_k(x-y,t)\{u_0^k(y)-(u_0^n)^k(y)\}\,dy\\
\le& C k^Ne^{-k^2 t}+ C_\tau \int \{u_0(y)-u_0^n(y) \} \,,dy
\end{aligned}
$$

Then, if $t\ge\tau>0$  we find,

$$
\begin{aligned}
U_n(x,t)\le \liminf_{k\to\infty}u^k(x,t)&\le \limsup_{k\to
\infty}u^k(x,t)\\
&\le U_n(x,t)+C_\tau \int\{u_0(y)-u_0^n(y)\}\,dy\\
&\le U_n(x,t)+\ep\qquad\mbox{if }n\ge n_0(\tau,\ep).
\end{aligned}
$$

Therefore, if there exists $\lim_{n\to\infty}U_n$ there holds that there exists $U=\lim_{k\to\infty}u^k=\lim_{n\to\infty}U_n$. Moreover, it is easy to prove from the estimates above that this limit is uniform on compact sets of $\R^N\times(0,\infty)$.

\medskip

On the other hand, by similar arguments we find that for $m>n$,
$$
\begin{aligned}
 0\le &\int \{U_m(x,t)-U_n(x,t) \}\,dx\\
 \le &\limsup_{k\to\infty} e^{-k^2 t}\int \{u_0^m(x)-u_0^n(x)\}\,dx+  \int\{u_0^m(y)-u_0^n(y)\}\,dy \to 0\quad\mbox{as } m,n\to\infty.
\end{aligned}
$$

Therefore, $M_n=\int U_n(x,t)\,dx\to M$ and, since $U_n=M_n U_\a$, we deduce that $ U=M U_\a$.

\medskip

So, it only remains to prove that $M=\int u_0(y)\,dy-\int_0^\infty\int u^p(y,t)\,dy\,dt$.

The estimates \eqref{masa}, \eqref{masaK}  hold in the present situation and, instead of the estimate  \eqref{masaKc} on the mass of $u^k$ and $U$ outside a large ball, we get
\begin{equation}\label{3.3}
\begin{aligned}
 \int_{|x|>K}u^k(x,t)\,dx\le& \int_{|x|>Kk}u_0(x)\,dx+ \int_{|x|>K}\int_{|y|<K/2}W_k(x-y,t)u_0^k(y)\,dy\,dx\\
+& \int_{|x|>K}\int_{|y|>K/2}W_k(x-y,t)u_0^k(y)\,dy\,dx\\
\le &  \int_{|x|>Kk}u_0(x)\,dx+ \int_{|x|>K}\int_{|y|<K/2}C\frac t{|x-y|^{N+2}}u_0^k(y)\,dy\,dx\\
+&\int_{|y|>K/2}u_0^k(y) \int_{|x|>K}W_k(x-y,t)\,dx\,dy\\
\le &  \int_{|x|>Kk}u_0(x)\,dx+C t \int_{|x|>K}\frac 1{|x|^{N+2}}\int_{|y|<Kk/2}u_0(y)\,dy\,dx\\
+&\int_{|y|>Kk/2}u_0(y)\,dy <\ep\quad\mbox{if }K\mbox{is large independently of }k\ge1
\end{aligned}\end{equation}

So, that we also have
$$\int_{|x|>K} U(x,t)\,dx\le \liminf_{k\to\infty}\int_{|x|>K}  u^k(x,t)\,dx\le \ep$$ if $K$ is large. So,
we get again that
$$M=\int U(x,t)\,dx= \int u_0(x)\,dx-\int_0^\infty\int u^p(x,t)\,dx\,dt
$$
 and the proposition is proved by using again that, in the critical case, $\int U(x,t)\,dx$ cannot be constant unless $U$ is identically zero.
\end{proof}

\bigskip

\section{Main results}
\setcounter{equation}{0}
In this section we prove our main results. Namely, we establish the asymptotic behavior
of solutions to
\begin{equation}\label{eq4}
\begin{cases}
u_t(x,t)=\int J(x-y)(u(y,t)-u(x,t))\,dy - u^p(x,t)\quad & \mbox{in }\,\R^N\times(0,\infty)\\
u(x,0) =  u_0(x)& \mbox{in }\,\R^N.
\end{cases}
\end{equation}
depending on the values of $p$, $N$ and the coefficient $\alpha$ of
non-integrability of the initial data $u_0$. At the end of this
section we also address the case where $u_0$ is integrable and
bounded. As mentioned before, we
have to distinguish between several cases.

We begin with the case $0<\alpha<N$. We further divide the analysis between the
supercritical case $p>1+2/\alpha$ and the critical case
$p=1+2/\alpha$. In the previous paper \cite{TW}, we were able to
establish the asymptotic behavior only for the supercritical
case. Now we present a different proof that actually allows us to
obtain the result both in the critical and supercritical cases. Moreover, in the
supercritical case, we prove the same type of result for
 more general non-integrable initial data.

\begin{teo} Let $u_0\in L^\infty$ be such that $|x|^\alpha u_0(x)\to A>0$ as $|x|\to\infty$ with  $0<\alpha<N$.  Let $p\ge 1+2/\alpha$ and $u$  the solution to \eqref{eq4}. Then, for every $R>0$,
$$
t^{\alpha/2} |u(x,t)-U(x,t)| \to0\quad\mbox{as }t\to\infty\quad\mbox{uniformly in }|x|\le R\sqrt t
$$
where  $U$ is the solution to
\begin{equation}\label{eqUa}
\left\{\begin{aligned}
&U_t-\a\Delta U=-c_0 U^p\\
&U(x,0)=  \frac {C_{A,N}}{|x|^\alpha}
\end{aligned}\right.
\end{equation}
with $c_0=0$ if $p>1+2/\alpha$ and $c_0=1$ if $p=1+2/\alpha$.

In a similar way we get the following results: When $\frac{|x|^\alpha}{\log |x|} u_0(x)\to A>0$ as $|x|\to\infty$ with $ 0<\alpha<N$ and $p>1+2/\alpha$, for every $R>0$ there holds that
$$
{t^{\alpha/2}}\Big|\frac{ u(x,t)}{\log t^{1/2}}-U(x,t)\Big| \to0\quad\mbox{as }t\to\infty\quad\mbox{uniformly in }|x|\le R\sqrt t
$$
with $U$ as above.

When $|x|^\alpha(\log|x|)u_0(x)\to A>0$ as $|x|\to\infty$ with $0<\alpha<N$ and $p\ge 1+2/\alpha$, for every $R>0$ there holds that
$$
{t^{\alpha/2}}\Big|u(x,t)\log t^{1/2}-U(x,t)\Big| \to0\quad\mbox{as }t\to\infty\quad\mbox{uniformly in }|x|\le R\sqrt t
$$
with $U$ the solution to \eqref{eqUa} with $c_0=0$ even in the critical case $p=1+2/\alpha$.

\end{teo}
\begin{proof} We know that the family $u^k$ is precompact in $C(\mathcal K)$ for every compact set $\mathcal K\subset \R^N\times(0,\infty)$. Therefore, for every sequence $k_n\to \infty$ there exists a subsequence that we still call $k_n$ such that $u^{k_n}$ converges uniformly on every compact subset of $ \R^N\times(0,\infty)$ to a function $U$.

On the other hand, it is easy to see that $u_0^k\to \frac {C_{A,N}}{|x|^\alpha}$ in the sense of distributions. Moreover, we are in the situation of
 Proposition \ref{determineUa}. So that, $U$ is the solution to \eqref{eqUa}. Thus, the whole family $u^k$ converges to $U$ as $k\to \infty$.
 In particular, $u^k(x,1)\to U(x,1)$ uniformly on compact sets of $\R^N$.

 As the solution of \eqref{eqUa} is invariant under the present rescaling. There holds that $U(y,1)=U^k(y,1)$. Thus, for every $R>0$,
 $$
 k^\alpha|u(ky,k^2)-U(ky,k^2)|\to0\quad\mbox{uniformly for }|y|\le R.
 $$

 By calling $x=y\sqrt t$, $t=k^2$ we get the result.

 \bigskip

 When $\frac{|x|^\alpha}{\log |x|} u_0(x)\to A>0$ as $|x|\to\infty$ with $0<\alpha<N$, it is easy to see that we still have  that $u_0^k\to  \frac {C_{A,N}}{|x|^\alpha}$ in the sense of distributions.  So the result follows also in this case.

 Analogously, when $|x|^\alpha(\log|x|)u_0(x)\to A>0$ as $|x|\to\infty$ with $0<\alpha<N$, we also have that $u_0^k\to  \frac {C_{A,N}}{|x|^\alpha}$ in the sense of distributions, and the result follows.
\end{proof}

\bigskip

We now analyze the case $\alpha=N$. Once again, we prove the result for more general non-integrable
initial data than the one considered in \cite{TW}.

\begin{teo}  Let $u_0\in L^\infty$ be such that $|x|^N u_0(x)\to A>0$ as $|x|\to\infty$. Let $p> 1+2/N$. Then, for every $R>0$,
$$
{t^{N/2}}\Big |\frac{u(x,t)}{\log t^{1/2}}-U(x,t)\Big| \to0\quad\mbox{as }t\to\infty\quad\mbox{uniformly in }|x|\le R\sqrt t
$$
where $u$ is the solution to \eqref{eq4} and $U$ is the solution to
$$
\left\{\begin{aligned}
&U_t-\a\Delta U=0\\
&U(x,0)= C_{A,N}\delta
\end{aligned}\right.
$$
with $\delta$ Dirac's delta.

In a similar way we get the following result when $\frac{|x|^N}{\log |x|} u_0(x)\to A>0$ as $|x|\to\infty$ and $p>1+2/N$: for every $R>0$,
$$
{t^{N/2}}\Big |\frac{u(x,t)}{\log^2 t^{1/2}}-U(x,t)\Big| \to0\quad\mbox{as }t\to\infty\quad\mbox{uniformly in }|x|\le R\sqrt t
$$
with $U$ as above.

When $|x|^N(\log |x|) u_0(x)\to A>0$ as $|x|\to\infty$ and $p>1+2/N$ we get: for every $R>0$,
$$
{t^{N/2}}\Big |\frac{u(x,t)}{\log\log t^{1/2}}-U(x,t)\Big| \to0\quad\mbox{as }t\to\infty\quad\mbox{uniformly in }|x|\le R\sqrt t
$$
with $U$ as above.
\end{teo}
\begin{proof} We proceed as in the previous theorem. This time, $u_0^k\to C_{A,N}\delta$ in the sense of distributions.
In fact,
$$
\int u_0^k(x)\varphi(x)\,dx= \frac1{\log k}\int_{|x|<Kk} u_0(x)\varphi(x/k)\,dx
$$
where $K$ is such that $\varphi(x)=0$ if $|x|>K$. Since $\varphi(x/k)\to\varphi(0)$ uniformly in $\R^N$ the result easily follows from this formula and the fact that
$$
 \frac1{\log K k}\int_{|x|<Kk} u_0(x)\,dx\to C_{A,N},\qquad \frac{\log Kk}{\log k}\to1 \qquad\mbox{as }k\to\infty.
 $$

The other cases follow similarly.
\end{proof}

\bigskip

Finally we consider the case where $u_0$ is
integrable and bounded. Such case has been studied by Pazoto and Rossi
in \cite{PR} for non-critical values of $p$ . Here we present
a new proof the includes the critical case and therefore settles the
question as far as integrable data is concerned.

\begin{teo} Let $u_0\in L^1\cap L^\infty$ and $p\ge 1+2/N$. Let $u$ be the solution to \eqref{eq4}.

First, assume $p>1+2/N$. Then, for every $R>0$,
$$
t^{N/2} |u(x,t)-U(x,t)| \to0\quad\mbox{as }t\to\infty\quad\mbox{uniformly in }|x|\le R\sqrt t
$$
where  $U$ is the solution to
\begin{equation}\label{eqUN}
\left\{\begin{aligned}
&U_t-\a\Delta U=0\\
&U(x,0)= M\delta
\end{aligned}\right.
\end{equation}
with $\delta$ Dirac's delta and $M=\int u_0(x)\,dx-\int_0^\infty\int u^p(x,t)\,dx\,dt$.

Now, let $p=1+2/N$. For every $R>0$ there holds that
$$
t^{N/2} u(x,t)\to0\quad\mbox{as }t\to\infty\quad\mbox{uniformly in }|x|\le R\sqrt t
$$

\end{teo}
\begin{proof}
As in the previous theorems we know that $u^k(y,1)\to U(y,1)$ uniformly on compact sets of $\R^N\times(0,\infty)$ as $k\to\infty$.
In the present situation, if $p>1+2/N$,  we know that  $U$ is the solution to \eqref{eqUN}. In particular, $U$ is invariant under the present rescaling. When $p=1+2/N$ we know that $U\equiv0$.
So, we get the result as in the previous theorems.
\end{proof}

\bigskip

\begin{rem} The same method allows to study the asymptotic behavior of the solution $u_L$ of the equation without absorption.
\end{rem}


\begin{thebibliography}{BHRP}

\bibitem{BCh1} P. Bates, A. Chmaj, {\it An integrodifferential model for phase transitions: Stationary solutions in higher dimensions}, J. Statistical Phys. {\bf 95}, 1999, 1119--1139.

\bibitem{BCh2} P. Bates, A. Chmaj, {\it A discrete convolution model for phase transitions}, Arch. Rat. Mech. Anal. {\bf 150}, 1999, 281--305.

\bibitem{BFRW} P. Bates, P. Fife, X. Ren, X. Wang, {\it Travelling waves in a convolution model for phase transitions},  Arch. Rat. Mech. Anal. {\bf 138}, 1997, 105--136.

\bibitem{CF} C. Carrillo, P. Fife, {\it Spatial effects in discrete generation population models}, J. Math. Biol. {\bf 50(2)}, 2005, 161--188.

\bibitem{ChChR} M. Chaves, E. Chasseigne, J. D. Rossi, {\it Asymptotic behavior for nonlocal diffusion equations}, Adv. Differential Equations, {\bf 2}, 2006, 271--291.

    \bibitem{CQW} X. Chen, Y. W. Qi, M. Wang, {\it Long time behavior of solutions to p-laplacian equation with absorption}, SIAM Jour. Math. Anal. {\bf 35(1)}, 2003, 123--134.

\bibitem{CEQW} C. Cortazar, M. Elgueta, F. Quiros, N. Wolanski, {\it Large time behavior of the solution to the Dirichlet problem for a nonlocal diffusion equation in an exterior domain}, in preparation.

\bibitem{CER} C. Cortazar, M. Elgueta, J. D. Rossi, {\it Nonlocal diffusion problems that approximate the heat equation with Dirichlet boundary conditions},  Israel Journal of Mathematics. {\bf 170(1)}, 2009, 53-60.

\bibitem{CERW2} C. Cortazar, M. Elgueta, J. D. Rossi, N. Wolanski, {\it How to approximate the heat equation with Neumann boundary conditions by nonlocal diffusion problems}, Arch. Rat. Mech. Anal. {\bf 187(1)}, 2008, 137--156.

\bibitem{F} P. Fife, {\it Some nonclassical trends in parabolic and parabolic-like evolutions}, Trends in nonlinear analysis, 153--191, Springer, Berlin, 2003.

\bibitem{GO} G. Gilboa, S. Osher, {\it Nonlocal operators with application to image processing}, Multiscale Model. Simul., {\bf 7(3)}, 2008, 1005, 1028.

\bibitem{H} L. Herraiz, {\it Asymptotic behavior of solutions of some semilinear parabolic problems}, Ann. Inst. Henri Poincare, {\bf 16(1)}, 1999, 49--105.

\bibitem{IR} Ignat, J. D. Rossi, {\it Refined asymptotic expansions for nonlocal diffusion equations}, J.  Evolution Equations. {\bf 8}, 2008, 617--629.

\bibitem{KP} S. Kamin, L. A. Peletier, {\it Large time behavior of solutions of the heat equation with absorption}, Anal. Scuola. Norm. Sup. Pisa Serie 4, {\bf 12}, 1985, 393--408.

\bibitem{KP2} S. Kamin, L. A. Peletier, {\it Large time behavior of
  solutions of the porous media equation with absorption}, Israel
  J. Math., {\bf 55}, 1986, 129--146.



\bibitem{KU} S. Kamin, M. Ughi, {\it On the behavior as $t\to\infty$ of the solutions of the Cauchy problem for certain nonlinear parabolic equations}, J. Math. Anal. Appl. {\bf 128}, 1997, 456--469.

\bibitem{LW} C. Lederman, N. Wolanski, {\it Singular perturbation in a nonlocal diffusion model},
Communications in PDE {\bf 31(2)}, 2006, 195--241.

\bibitem{PR} A. Pazoto, J. D. Rossi, {\it Asymptotic behavior for a semilinear nonlocal equation}. Asymptotic Analysis. {\bf 52(1-2)}, 2007, 143--155.

\bibitem{TW} J. Terra, N. Wolanski, {\it Asymptotic behavior for a nonlocal diffusion equation with absorption and nonintegrable initial data. The supercritical case}, submitted.


\bibitem{Z} L. Zhang, {\it Existence, uniqueness and exponential stability of traveling wave solutions of some integral differential equations arising from neuronal networks}, J. Differential Equations {\bf 197(1)}, 2004, 162--196.

    \bibitem{Za} J. Zhao, {\it The Asymptotic Behavior of solutions of a quasilinear degenerate parabolic equation},
J. Differential Equations, {\bf 102}, 1993,  33-–52.


\end{thebibliography}
\end{document}